\documentclass[a4paper,12pt]{article}
\usepackage{latexsym}
\usepackage{amsmath}
\usepackage{amssymb}
\usepackage{enumerate}
\usepackage{bm}
\usepackage{xcolor}
\usepackage{mathrsfs}
\usepackage[anchorcolor=blue,%
 bookmarks=true,%
 bookmarksnumbered=true,%
]{hyperref}

\usepackage{theorem}
\newtheorem{theorem}{Theorem}[section]
\newtheorem{proposition}[theorem]{Proposition}
\newtheorem{lemma}[theorem]{Lemma}

\theorembodyfont{\rmfamily}
\newtheorem{proof}{\textmd{\textit{Proof.}}}

\newtheorem{remark}[theorem]{Remark}
\newtheorem{example}[theorem]{Example}
\newtheorem{definition}[theorem]{Definition}

\makeatletter

\@addtoreset{equation}{section}
\makeatother

\newcommand{\qedd}{\hfill \Box}
\newcommand{\ve}{\varepsilon}

\newcommand{\lra}{\longrightarrow}
\newcommand{\e}{\mathrm{e}}

\newcommand{\N}{\ensuremath{\mathbb{N}}}

\newcommand{\R}{\ensuremath{\mathbb{R}}}

\newcommand{\bbD}{\ensuremath{\mathbb{D}}}
\newcommand{\cE}{\ensuremath{\mathcal{E}}}
\newcommand{\cL}{\ensuremath{\mathcal{L}}}
\newcommand{\cP}{\ensuremath{\mathcal{P}}}
\newcommand{\cU}{\ensuremath{\mathcal{U}}}
\newcommand{\cV}{\ensuremath{\mathcal{V}}}
\newcommand{\sH}{\ensuremath{\mathsf{H}}}
\newcommand{\fm}{\ensuremath{\mathfrak{m}}}

\def\div{\mathop{\mathrm{div}}\nolimits}
\def\supp{\mathop{\mathrm{supp}}\nolimits}

\def\loc{\mathop{\mathrm{loc}}\nolimits}
\def\Ent{\mathop{\mathrm{Ent}}\nolimits}
\def\Hess{\mathop{\mathrm{Hess}}\nolimits}
\def\Ric{\mathop{\mathrm{Ric}}\nolimits}

\def\CD{\mathop{\mathrm{CD}}\nolimits}
\def\RCD{\mathop{\mathrm{RCD}}\nolimits}
\def\HS{\mathop{\mathrm{HS}}\nolimits}
\def\ac{\mathop{\mathrm{ac}}\nolimits}

\def\Ch{\mathop{\mathsf{Ch}}\nolimits}

\title{Rigidity for the spectral gap on $\RCD(K,\infty)$-spaces}
\author{Nicola Gigli\thanks{SISSA, Trieste,
Italy ({\sf ngigli@sissa.it})} $\cdot$
Christian Ketterer\thanks{University of Toronto,
Canada ({\sf ckettere@math.toronto.edu})} $\cdot$
Kazumasa Kuwada\thanks{Mathematical Institute, Tohoku University,
Sendai 980-8578, Japan ({\sf kuwada@m.tohoku.ac.jp})} $\cdot$
Shin-ichi Ohta\thanks{Department of Mathematics, Osaka University,
Osaka 560-0043, Japan ({\sf s.ohta@math.sci.osaka-u.ac.jp})}}
\date{\today}
\begin{document}

\maketitle

\begin{abstract}
We consider a rigidity problem for the spectral gap of the Laplacian on an $\RCD(K,\infty)$-space
(a metric measure space satisfying the Riemannian curvature-dimension condition) for positive $K$.
For a weighted Riemannian manifold,
Cheng--Zhou showed that the sharp spectral gap is achieved
only when a $1$-dimensional Gaussian space is split off.
This can be regarded as an infinite-dimensional counterpart to Obata's rigidity theorem.
Generalizing to $\RCD(K,\infty)$-spaces is not straightforward
due to the lack of smooth structure and doubling condition.
We employ the lift of an eigenfunction to the Wasserstein space
and the theory of regular Lagrangian flows
recently developed by Ambrosio--Trevisan to overcome this difficulty.
\end{abstract}
\tableofcontents
\section{Introduction}\label{sc:intro}%%%%%%%%%%%
%%%%%%%%%%%%%%%%%%%%%

The \emph{Riemannian curvature-dimension condition} $\RCD(K,N)$
is a synthetic notion of lower Ricci curvature bound for metric measure spaces
(roughly speaking, $K$ means a lower Ricci curvature bound
and $N$ acts as an upper dimension bound).
After its birth in \cite{AGS-rcd} for $N=\infty$, 
and further developments 
for $N<\infty$ from \cite{Gi-Ondiff,Gi-split} to \cite{AMS2,EKS}, 
the theory of metric measure spaces satisfying $\RCD(K,N)$
(called \emph{$\RCD(K,N)$-spaces} for short)
has been making a breathtaking progress.
There is already a long list of achievements,
including the Laplacian comparison theorem,
%\cite{Gi-Ondiff}
the splitting theorem of Cheeger--Gromoll type \cite{Gi-split}
and the isoperimetric inequality of L\'evy--Gromov type \cite{CM}.
Very recently, Cavalletti--Milman \cite{CMi} showed that
the \emph{$\RCD^*(K,N)$-condition}, which is defined as a variant of the $\RCD(K,N)$-condition,
is in fact equivalent to the $\RCD(K,N)$-condition.
The aim of the present article is to add a rigidity result
on the spectral gap of $\RCD(K,\infty)$-spaces with $K>0$ to this list,
as an application of the recently developed theories
on regular Lagrangian flows (\cite{AT}), the splitting theorem (\cite{Gi-split}),
and on the relation between the Hessian and the convexity of functions (\cite{Ke-obata}).

In an $\RCD(K,\infty)$-space $(X,d,\fm)$ with $K>0$,
we have the \emph{spectral gap} $\lambda_1 \ge K$
for the first nonzero eigenvalue of the Laplacian,
in other words, the (global) \emph{Poincar\'e inequality}
\begin{equation}\label{eq:Poin}
\int_X f^2 \,d\fm -\bigg( \int_X f \,d\fm \bigg)^2
 \le \frac{1}{K} \int_X |\nabla f|^2 \,d\fm
\end{equation}
holds for all $f \in W^{1,2}(X)$.
For $\RCD(K,N)$-spaces with $K>0$ and $N \in (1,\infty)$,
one can improve the above spectral gap to the \emph{Lichnerowicz inequality}
$\lambda_1 \ge KN/(N-1)$ \cite{EKS}[Theorem 4.22]. Moreover, for $CD(K,N)$-spaces the same estimate was obtained in \cite{LV-de}.
In \cite[Theorem~1.2]{Ke-obata}, \emph{Obata's rigidity theorem}
in Riemannian geometry (\cite{Ob}) was generalized to $\RCD(K,N)$-spaces as follows:
If an $\RCD(N-1,N)$-space $(X,d,\fm)$ with $N \in [2,\infty)$ satisfies the sharp gap $\lambda_1=N$,
then $(X,d,\fm)$ is represented as the spherical suspension of an $\RCD(N-2,N-1)$-space.
Note that assuming $K=N-1$ does not lose any generality
thanks to the scaling property of the $\RCD$-condition,
and see \cite{Ke-obata} for the cases of $N \in (1,2)$ and $N=1$.
We remark that $\lambda_1=N$ is achieved by a smooth weighted Riemannian manifold (without boundary)
satisfying $\Ric_N \ge N-1$ only when $N=\dim M$ and $M$ is isometric to the unit sphere
(see \cite[Theorem~1.1]{Ku}, where drifts of non-gradient type are also considered).

Our main theorem can be regarded as the infinite-dimensional counterpart
to the above generalized Obata theorem.
Briefly speaking, if the eigenvalue achieves its minimum $K$ with multiplicity $k$,
then $(X,d,\fm)$ splits off the $k$-dimensional Gaussian space.
We remark that, on an $\RCD(K,\infty)$-space with $K>0$,
the embedding of $W^{1,2}(X)$ into $L^2(X)$ is compact 
(\cite[Proposition~6.7]{GMS}), 
hence the Laplacian has the discrete spectrum (with finite multiplicities)
that we denote by $\sigma ( - \Delta ) = \{\lambda_i\}_{i = 0}^\infty$ 
with $\lambda_i \le \lambda_{i+1}$ and $\lambda_0 = 0$.

\begin{theorem}\label{th:main}
Let $(X,d,\fm)$ be an $\RCD(K,\infty)$-space with $K>0$,
and assume that $\lambda_i=K$ holds for $1 \le i \le k$.
Then there exists an $\RCD(K,\infty)$-space $(Y,d_Y,\fm_Y)$ such that$:$
\begin{enumerate}[{\rm (i)}]
\item
The metric space $(X,d)$ is isometric to the product space $(Y,d_Y) \times (\R^k,|\cdot|)$
with the $L^2$-product metric, where $|\cdot|$ is the Euclidean norm/distance.

\item
Through the isometry above,
the measure $\fm$ coincides with the product measure
$\fm_Y \times \e^{-K|x|^2/2} dx^1 \cdots dx^k$,
where $dx^1 \cdots dx^k$ denotes the Lebesgue measure on $\R^k$.
\end{enumerate}
\end{theorem}

This rigidity was first shown on weighted Riemannian manifolds
by Cheng--Zhou \cite[Theorem~2]{CZ}.
(The first assertion (i) on the isometric splitting was also (informally) pointed out in \cite[p.~1547]{HN}
as an outcome of the improved Bochner inequality in \cite{BE}.)
See \cite{Mai} for a recent extension to the case of negative effective dimension ($N<0$).
It is worthwhile to review the proof in \cite{CZ}.
Let $u$ be an eigenfunction for the sharp spectral gap $K$.
Then the Bochner inequality under $\Ric_{\infty} \ge K$ becomes equality for $u$,
and shows $\Hess u \equiv 0$ (in other words, $u$ is affine).
Therefore $\nabla u$ is a parallel vector field and the de Rham decomposition
provides the isometric splitting as in (i).
The behavior of the measure in (ii) is also deduced from the Bochner inequality.
This argument reminds us the proof of Cheeger--Gromoll's splitting theorem \cite{CG},
the role of the Busemann function in \cite{CG} is replaced by the eigenfunction $u$.
The splitting theorem was generalized to $\RCD(0,N)$-spaces in \cite{Gi-split},
thus it is natural to consider an analogue of \cite{CZ} for $\RCD(K,\infty)$-spaces.

Although the assertion of Theorem~\ref{th:main} is the same as the Riemannian case,
the generalization to $\RCD(K,\infty)$-spaces is technically challenging.
The lack of the smooth structure
(precisely, parallel vectors fields and the de Rham decomposition)
prevents us following the simple proof of \cite{CZ}.
Moreover, compared with \cite{Gi-split,Ke-obata} on $\RCD(K,N)$-spaces,
the absence of upper dimension bound causes several difficulties
(for instance, our measure $\fm$ is not necessarily doubling
and $X$ is not locally compact).
In order to overcome these difficulties,
we consider the lift $\cU$ of the eigenfunction $u$ to the $L^2$-Wasserstein space,
defined by $\cU(\mu):=\int_X u \,d\mu$.
We deduce from $\Hess u \equiv 0$ (almost everywhere) that
{$\cU$} is affine 
by generalizing the discussion in \cite{Ke-obata} (Theorem~\ref{th:affine}).
Then we employ the \emph{regular Lagrangian flow} of the negative gradient vector field
$-\nabla u$ of the eigenfunction $u$ (Theorem~\ref{th:RLF}),
and show that its lift gives the gradient flow of $\cU$
in the sense of the evolution variational \emph{equality} 
(Lemma~\ref{lem:w-EVI}).
These precise behaviors of $\cU$ allow us to go down to $u$,
and we eventually see that $u$ itself is affine (Proposition~\ref{pr:u-affine}).

The article is organized as follows.
Section~\ref{sc:prel} is devoted to the preliminaries for $\RCD(K,\infty)$-spaces.
We divide the proof of Theorem~\ref{th:main} into 4 sections.
In Section~\ref{sc:affine} we show that the lift $\cU$ of the eigenfunction $u$
is affine along the lines of \cite{Ke-obata}.
We then apply the theory in \cite{AT} to obtain the regular Lagrangian flow $(F_t)_{t \in \R}$
of $-\nabla u$ in Section~\ref{sc:gf} (this step is not straightforward since $u$ is unbounded),
and analyze the behaviors of the measure $\fm$ and the distance $d$ along the flow.
With these properties of the flow,
we can follow the argument in \cite{Gi-split} to prove the $k=1$ case of Theorem~\ref{th:main},
as we shall see in Section~\ref{sc:isom}.
In Section~\ref{sc:final} we complete the proof by iteration,
followed by some concluding remarks.
\medskip

A part of this joint work was done while NG, CK and KK
visited Kyoto University in September 2016,
on the occasion of the RIMS International Research Project
``Differential Geometry and Geometric Analysis''.
The authors thank RIMS for its hospitality.
NG was supported in part by the MIUR SIR-grant `Nonsmooth Differential Geometry' (RBSI147UG4).
KK was supported in part by JSPS Grant-in-Aid for Young Scientist (KAKENHI) 26707004.
SO was supported in part by JSPS Grant-in-Aid for Scientific Research (KAKENHI) 15K04844.

\section{Preliminaries for $\RCD$-spaces}\label{sc:prel}%%%%%%%%%%%
%%%%%%%%%%%%%%%%%%%%%

In this section we review the definition and some properties of $\RCD$-spaces.
We refer to \cite{Vi} for the foundation of optimal transport theory and $\CD$-spaces, and 
\cite{AGS-rcd,AGMR,AMS2,EKS,Gi-Ondiff,Gi-split}
for the reinforced notion of $\RCD$-spaces.

\subsection{$\CD(K,\infty)$-spaces}\label{ssc:CD}%%%%%%%
%%%%%%%%%%%%%%%%%%

Let $(X,d)$ be a complete and separable metric space,
and $\fm$ be a Borel measure on $X$ which is finite on bounded sets.
We in addition assume that $(X,d)$ is a \emph{geodesic} space
in the sense that every pair $x,y \in X$ is connected by a \emph{minimal geodesic}
$\gamma:[0,1] \lra X$ such that $\gamma(0)=x$, $\gamma(1)=y$ and
$d(\gamma(s),\gamma(t))=|s-t|d(x,y)$ (all geodesics in this paper will be minimal).

Denote by $\cP(X)$ the space of Borel probability measures on $X$,
and by $\cP^2(X) \subset \cP(X)$ the subset consisting of measures with finite second moment.
The $L^2$-\emph{Wasserstein distance} on $\cP^2(X)$ will be denoted by $W_2$.
We denote by $\cP^2_{\ac}(X) \subset \cP^2(X)$ the subset
consisting of absolutely continuous measures with respect to $\fm$ ($\mu \ll \fm$).
We recall a basic fact in optimal transport theory for later convenience.
For $\mu,\nu \in \cP^2(X)$, the \emph{Kantorovich duality}
\begin{equation}\label{eq:Kant}
\frac{W_2^2(\mu,\nu)}{2} =\sup_{(\varphi,\psi)}
 \bigg\{ \int_X \varphi \,d\mu -\int_X \psi \,d\nu \,\bigg|\,
 \varphi(x)-\psi(y) \le \frac{d^2(x,y)}{2} \bigg\}
\end{equation}
holds, and a pair $(\varphi,\psi)$ attaining the above infimum is called
a \emph{Kantorovich potential} for $(\mu,\nu)$.
Kantorovich potentials are given by locally Lipschitz functions 
under mild assumptions.
%Given a $W_2$-minimal geodesic $(\mu)_{t \in [0,1]}$ from $\mu$ to $\nu$,
%a kind of \emph{superposition principle} holds in the sense that
%there exists a probability measure $\Pi$ on the space $\mathrm{Geo}(X)$
%of minimal geodesics $\gamma:[0,1] \lra X$ such that
%\begin{equation}\label{eq:Pi}
%(e_0)_*\Pi=\mu, \qquad (e_1)_*\Pi=\nu, \qquad
%\int_{\mathrm{Geo}(X)} d^2 \big( \gamma(0),\gamma(1) \big) \,\Pi(d\gamma)
%\end{equation}
%(see \cite[Theorem~4.2]{Li}).
%This measure $\Pi$ is sometimes called a \emph{dynamical optimal coupling}
%(or a \emph{dynamical optimal transference plan})
%associated with $(\mu_t)_{t \in [0,1]}$.

Now we turn to the definition of $\CD$-spaces.
For $\mu \in \cP^2(X)$, the \emph{relative entropy} with respect to $\fm$
is defined by
\[ \Ent_{\fm}(\mu):=\int_X \rho\log\rho \,d\fm \]
if $\mu=\rho \fm \in \cP^2_{\ac}(X)$ and $\int_{\{\rho>1\}} \rho\log\rho \,d\fm<\infty$,
otherwise $\Ent_{\fm}(\mu):=\infty$.

\begin{definition}[Curvature-dimension condition]\label{df:CD}
Let $K \in \R$.
We say that $(X,d,\fm)$ satisfies the \emph{curvature-dimension condition} $\CD(K,\infty)$
(or $(X,d,\fm)$ is a \emph{$\CD(K,\infty)$-space}) if $\Ent_{\fm}$ is \emph{$K$-convex}
in the sense that, for any $\mu_0,\mu_1 \in \cP^2(X)$,
there is a minimal geodesic $(\mu_t)_{t \in [0,1]}$ between them with respect to $W_2$
such that
\begin{equation}\label{eq:Kcon}
\Ent_{\fm}(\mu_t) \le (1-t)\Ent_{\fm}(\mu_0) +t\Ent_{\fm}(\mu_1)
 -\frac{K}{2}(1-t)t W_2^2(\mu_0,\mu_1)
\end{equation}
for all $t \in (0,1)$.
\end{definition}

One can moreover define $\CD(K,N)$ for $K \in \R$ and $N \in [1,\infty]$,
and then $\CD(K,N)$ is equivalent to the combination `$\Ric \ge K$ and $\dim \le N$'
for Riemannian manifolds equipped with the Riemannian volume measures (\cite{vRS,StI,StII,LV}).
This characterization is extended to weighted Riemannian and Finsler manifolds
by means of the \emph{weighted Ricci curvature} $\Ric_N$,
namely $\CD(K,N)$ is equivalent to $\Ric_N \ge K$ (\cite{StI,StII,LV,Oh-int}).
Recently there are further generalizations to
the cases of $N<0$ as well as $N=0$ (\cite{Oh-neg,Oh-needle}).

A particularly important example relevant to our result is the following.

\begin{example}[Gaussian spaces]\label{ex:Gauss}
Consider a weighted Euclidean space $(\R^n,|\cdot|,\e^{-\psi} dx)$,
where $\psi \in C^{\infty}(\R^n)$.
Then we have $\Ric_{\infty} =\Hess \psi$,
and hence $(\R^n,|\cdot|,\e^{-\psi} dx)$ satisfies $\CD(K,\infty)$
if and only if $\psi$ is $K$-convex ($\Hess \psi \ge K$).
For instance, the \emph{Gaussian space} $(\R^n,|\cdot|,\e^{-K|x|^2/2}dx)$
is a $\CD(K,\infty)$-space, regardless of the dimension $n$.
\end{example}

Let us recall two fundamental properties of $\CD(K,\infty)$-spaces for later convenience.

\begin{lemma}[Properties of $\CD$-spaces]\label{lm:CD}
Let $(X,d,\fm)$ be a $\CD(K,\infty)$-space.
\begin{enumerate}[{\rm (i)}]
\item For $a,b>0$, the scaled space $(X,a \cdot d,b \cdot \fm)$
satisfies $\CD(K/a^2,\infty)$.

\item If $K>0$, then the measure $\fm$ has the Gaussian decay$:$
\[ \fm\big( B_r(x) \setminus B_{r-\ve}(x) \big) \le C_1 \e^{-K(r-C_2)^2/2} \]
for some positive constants $C_i=C_i(K,\ve)$, $i=1,2$, and for $r \gg \ve$.
In particular, we have $\fm(X)<\infty$.
\end{enumerate}
\end{lemma}

Notice that (i) is immediate from the definition.
See \cite[Theorem~4.26]{StI} for (ii).

In order to develop analysis on $\CD(K,\infty)$-spaces,
we introduce the \emph{Cheeger energy} (named after \cite{Ch})
for $f \in L^2(X)$ as
\[ \Ch(f):=\frac{1}{2} \inf_{\{f_i\}_{i \in \N}} \liminf_{i \to \infty}
 \int_X |\nabla^L f_i|^2 \,d\fm, \]
where $\{f_i\}_{i \in \N}$ runs over all sequences of Lipschitz functions
such that $f_i \to f$ in $L^2(X)$, and
\[ 
|\nabla^L h|(x) :=\limsup_{y \to x} \frac{|h (y)- h (x)|}{d(x,y)} 
\]
for $h :X \lra \R$. 
We define the associated \emph{Sobolev space} by
\[ W^{1,2}(X):=\{ f \in L^2(X) \,|\, \Ch(f)<\infty \}. \]
Given $f \in W^{1,2}(X)$, there exists the unique \emph{minimal weak upper gradient}
$|\nabla f| \in L^2(X)$ such that
\[ \Ch(f)=\frac{1}{2} \int_X |\nabla f|^2 \,d\fm. \]
We refer to \cite{Ch,Sha,AGS-hf} for further discussions.

When $K>0$, a $\CD(K,\infty)$-space enjoys the Poincar\'e inequality \eqref{eq:Poin}
mentioned in the introduction, as well as
the log-Sobolev and Talagrand inequalities (see \cite[\S 6]{LV}).

\subsection{$\RCD(K,\infty)$-spaces}\label{ssc:RCD}%%%%%%%
%%%%%%%%%%%%%%%%%%

As we mentioned after Definition~\ref{df:CD},
the curvature-dimension condition does not rule out Finsler manifolds.
On the one hand, this was a starting point of the rich theory
of the weighted Ricci curvature of Finsler manifolds
(see \cite{Oh-int} and the recent survey \cite{Oh-nlga}).
On the other hand, admitting Finsler manifolds (and especially normed spaces)
causes some difficulties, for instance,
the Cheeger energy is not quadratic and the associated Laplacian is nonlinear.
For this reason, it is natural to expect a `Riemannian' version of
the curvature-dimension condition, and the following notion given in \cite{AGS-rcd}
turned out successful and has been a subject of intensive research.

\begin{definition}[Riemannian curvature-dimension condition]\label{df:RCD}
For $K \in \R$,
we say that $(X,d,\fm)$ satisfies the \emph{Riemannian curvature-dimension condition}
$\RCD(K,\infty)$ (or $(X,d,\fm)$ is an \emph{$\RCD(K,\infty)$-space})
if it satisfies $\CD(K,\infty)$ and the Cheeger energy $\Ch$ is a quadratic form
in the sense that
\begin{equation}\label{eq:intH}
\Ch(f+g) +\Ch(f-g) =2\Ch(f) +2\Ch(g) \quad
 \text{for all}\ f,g \in W^{1,2}(X).
\end{equation}
\end{definition}

See \cite{AMS2,EKS,Gi-Ondiff,Gi-split} 
for the finite-dimensional counterpart $\RCD(K,N)$.
The quadratic property \eqref{eq:intH} is called the \emph{infinitesimal Hilbertianity}
and rules out (non-Riemannian) Finsler manifolds.
$\RCD(K,\infty)$-spaces enjoy several finer properties than $\CD(K,\infty)$-spaces.
For instance, the inequality \eqref{eq:Kcon} holds along every $W_2$-geodesics
(called the \emph{strong $K$-convexity}),
and any pair $\mu_0,\mu_1 \in \cP^2_{\ac}(X)$
is connected by a unique minimal geodesic.

Thanks to \eqref{eq:intH}, by polarization we can define
$\langle \nabla f,\nabla g \rangle \in L^1(X)$ for $f,g \in W^{1,2}(X)$ by
\[ \langle \nabla f,\nabla g \rangle
 :=\frac{1}{4} \big( |\nabla(f+g)|^2 -|\nabla(f-g)|^2 \big). \]
Then the bilinear form
\[ \cE(f,g):=\int_X \langle \nabla f,\nabla g \rangle \,d\fm,
 \qquad f,g \in W^{1,2}(X), \]
is a strongly local, quasi-regular Dirichlet form 
(see \cite[Section 6.2]{AGS-rcd}, cf.~\cite[Theorem~4.1]{Sa}),
and we call its generator $\Delta:D(\Delta) \lra L^2(X)$ the \emph{Laplacian},
which is a linear, self-adjoint, nonpositive definite operator such that
\[ \cE(f,\phi) =-\int_X \phi \cdot \Delta f \,d\fm,
 \qquad \phi \in W^{1,2}(X). \]
The domain $D(\Delta)$ is dense in $W^{1,2}(X)$ and $L^2(X)$.
We refer to \cite{BH,FOT} for the basic theory of Dirichlet forms.

We now review some connections between $| \nabla f |$ 
and the Lipschitz continuity of $f$ on $\RCD (K,\infty)$-spaces. 
We have in general $|\nabla f| \le |\nabla^L f|$ $\fm$-almost everywhere
for Lipschitz functions $f \in W^{1,2} (X)$.
If $(X,d,\fm)$ satisfies the volume doubling condition
and the local $(1,2)$-Poincar\'e inequality,
then $|\nabla f|=|\nabla^L f|$ holds $\fm$-almost everywhere 
for any Lipschitz function $f$ (see \cite{Ch}). 
In our framework of $\RCD(K,\infty)$-spaces, however,
both the doubling condition and the local Poincar\'e inequality may fail
(only weaker estimates such as Lemma~\ref{lm:CD}(ii)
as well as a sort of local Poincar\'e inequality in \cite{Ra} are available).
Nonetheless, we know that $f \in W^{1,2}(X)$ satisfying $|\nabla f| \le C$
for some $C \ge 0$ admits a $C$-Lipschitz {representative}
(see \cite[Theorem~6.2]{AGS-rcd} - this has been called the \emph{Sobolev-to-Lipschitz property} in
\cite[Theorem~6.8]{Gi-split}), 
and this fact is sufficient for our purpose.

The heat semigroup $(\sH_t)_{t \ge 0}$ associated with the Laplacian $\Delta$
enjoys various regularization properties.
For instance, the set $\mathcal{A} :=\bigcup_{t>0} \sH_t L^{\infty}(X)$
is dense both in $W^{1,2}(X)$ and $D(\Delta)$.
We also recall the following for later use.

\begin{proposition}[$L^{\infty}$-Lipschitz regularization]\label{pr:Lip}
If $f \in L^{\infty}(X)$, then $\sH_t f$ is Lipschitz for all $t>0$.
\end{proposition}

See \cite[Theorem~6.5]{AGS-rcd} (and \cite[Theorem~7.3]{AGMR})
for a quantitative estimate of the Lipschitz constant.
Moreover, $\sH_t$ can be extended canonically 
to a map from $\cP^2 (X)$ to itself, and
the \emph{$W_2$-contraction property} holds:
\begin{equation}\label{eq:Wcont}
W_2 \big( \sH_t(\mu),\sH_t(\nu) \big) \le \e^{-Kt} W_2(\mu,\nu), \qquad
 \mu,\nu \in \cP^2(X).
\end{equation}
This property \eqref{eq:Wcont} in fact characterizes $\RCD(K,\infty)$-spaces
among infinitesimally Hilbertian spaces, see \cite[Theorem~4.1]{Sa}
and \cite{AGS-rcd,AGS-be} for the precise statement.
It is worthwhile to mention that \eqref{eq:Wcont} fails in normed spaces
and Finsler manifolds (\cite{OS-nc}).

Set
\[ \bbD_{\infty}(X):=\{ f \in D(\Delta) \cap L^{\infty}(X) \,|\,
 |\nabla f| \in L^{\infty}(X),\ \Delta f \in W^{1,2}(X) \} \]
(which is denoted by $\text{TestF}(X)$ in \cite{Gi-dg}).
Note that $\mathcal{A} \subset \bbD_{\infty}(X)$.
In an $\RCD(K,\infty)$-space the \emph{Bochner inequality}
\begin{equation}\label{eq:Boch}
\frac{1}{2}\Delta(|\nabla f|^2) -\langle \nabla f,\nabla(\Delta f) \rangle
 \ge K|\nabla f|^2, \qquad f \in \bbD_{\infty}(X),
\end{equation}
holds in a weak sense (\cite{GKO,AGS-rcd,AGS-be}).
{The} {more precise definition of ``weak sense'' 
will be discussed in section~\ref{sc:affine}. }
Note that, for $f \in \bbD_{\infty}(X)$,
we have $|\nabla f|^2 \in W^{1,2}(X) \cap L^{\infty}(X)$ (\cite[Lemma~3.2]{Sa})
and hence (the continuous version of) $f$ is Lipschitz.
The inequality \eqref{eq:Boch} also characterizes $\RCD(K,\infty)$-spaces,
see \cite{AGS-be} and \cite[Theorem~4.1]{Sa}.

On Riemannian manifolds, the Bochner-Weitzenb\"ock formula yields that 
the left hand side is nothing but the sum of the Hilbert-Schmidt norm of 
$\Hess f$ and $\Ric (\nabla f, \nabla f )$. 
Thus it seems that the Hessian is missing in \eqref{eq:Boch}. 
Nevertheless, we are somehow able to recover it from \eqref{eq:Boch} 
by a so-called self-improvement technique going back to Bakry \cite{Ba1}. 
As we guess from the argument in \cite{CZ} on a Riemannian manifold, 
this self-improvement plays a fundamental role in the sequel. 
There are two different (but closely related) notions of the ``Hessian'' 
in this context. 
The one is $H[f]$ in terms of $\Gamma$-calculus 
and the other one is $\Hess f$ introduced in \cite{Gi-dg} 
on $\RCD (K,\infty)$ spaces. 
For $f \in \mathbb{D}_\infty (X)$, the former one 
$H [f] : \mathbb{D}_\infty (X) \times \mathbb{D}_\infty (X) \to L^2 (X)$
is defined as follows: 
\[
H[f](\phi,\psi) 
: = 
\frac{1}{2} \Big\{
 \big\langle \nabla\phi, \nabla \langle \nabla f,\nabla\psi \rangle \big\rangle
 +\big\langle \nabla\psi, \nabla \langle \nabla f,\nabla\phi \rangle \big\rangle
 -\big\langle \nabla f, \nabla \langle \nabla\phi,\nabla\psi \rangle \big\rangle \Big\}.
\]
The latter one, $\Hess f$, is more complicated and we omit the precise definition of it, 
since $\Hess f$ is a \emph{tensorial} object for vector fields unlike $H[f]$ 
(see \cite[Definition~3.3.1]{Gi-dg}). 
In smooth context, $\Hess f$ precisely coincides with the classical definition. 
For $f \in \mathbb{D}_\infty (X)$, we can define $\Hess f$ and it is identified with $H[f]$ 
in the following sense: For $\phi , \psi \in \mathbb{D}_\infty (X)$, 
we can define the associated gradient vector fields $\nabla \phi$ and $\nabla \psi$.
Then $\Hess f ( \nabla \phi , \nabla \psi )$ makes sense and 
it coincides with $H[f] (\phi, \psi)$ $\fm$-a.e.\ \cite[Theorem~3.3.8]{Gi-dg}. 
For a formal computation, a self-improvement involving $H[f]$ would be sufficient, 
but we need a stronger one involving $\Hess f$ to overcome technical difficulties 
(Note that such a difficulty arises from the fact that 
the eigenfunction does not belong to $\mathbb{D}_\infty (X)$). 
Indeed there are some advantages in working with $\Hess f$. 
Among others, $\Hess f$ is defined for $f \in W^{2,2} (X)$, 
where $W^{2,2} (X)$ is the second order Sobolev space 
(see \cite[Definition~3.3.1]{Gi-dg}). 
The only property of $W^{2,2} (X)$ we need in this article is 
$D ( \Delta ) \subset W^{2,2} (X)$ \cite[Corollary~3.3.9]{Gi-dg}. 
From this, we can see that $W^{2,2} (X)$ is much larger than $\mathbb{D}_\infty (X)$. 
As a tensorial object, we can define the Hilbert-Schmidt norm $| \Hess f |_{\HS}$ 
of $\Hess f$ and it gives an upper bound of the operator norm in the following sense: 
For $f, \phi, \psi \in \mathbb{D}_\infty (X)$, 
\[
| H [f] (\phi , \psi )| \le | \Hess f |_{\HS} | \nabla \phi | | \nabla \psi | 
\quad \mbox{$\fm$-a.e.}\ .
\]
The strongest self-improvement of \eqref{eq:Boch} 
involving $\Hess f$ in our framework is given 
as follows (see \cite[Theorem~3.3.8]{Gi-dg}):
\begin{equation}\label{eq:Sav}
\frac{1}{2}\Delta(|\nabla f|^2) -\langle \nabla f,\nabla(\Delta f)\rangle 
 \ge |\Hess f|^2_{\HS}+ K|\nabla f|^2
\end{equation}
in a weak sense for $f \in \bbD_{\infty}(X)$ 
(see \cite[Theorem~3.4]{Sa} for the weaker one involving $H[f]$ in our framework). 
One can in fact show \eqref{eq:Sav} in the $\fm$-almost everywhere sense
by replacing the left-hand side with the absolutely continuous part of 
the $\Gamma_2$-operator, see \cite{Sa,Gi-dg} for details.

\section{First step: Lift of the eigenfunction is affine}\label{sc:affine}%%%%%
%%%%%%%%%%%%%%%%%%%%%

We start the proof of Theorem~\ref{th:main}, divided into 4 steps. 
The first step is to show that the lift of the eigenfunction 
to the $L^2$-Wasserstein space is affine. 
We note that the statement does not follow from the recent result 
by Gigli and Tamanini on the second differentiation formula \cite{GT}
since such result is crucially based on finite dimensionality.

Recall from Lemma~\ref{lm:CD}(i) that we can normalize the curvature bound as $K=1$
without loss of generality.
Thus let $(X,d,\fm)$ be an $\RCD(1,\infty)$-space from here on.
Thanks to \cite[Proposition~6.7]{GMS},
the spectrum of the Laplacian is discrete and
the hypothesis $\lambda_1=1$ implies the existence of
an eigenfunction $u \in D(\Delta)$ satisfying $\Delta u=-u$.
Adapting the discussion in \cite{Ke-obata},
we will show that the lift $\cU$ of $u$ to $\cP^2(X)$, defined by
\begin{equation}\label{eq:U}
\cU(\mu) :=\int_X u \,d\mu, \qquad \mu \in \cP^2(X),
\end{equation}
is \emph{affine} (or \emph{totally geodesic}).
To be precise, we prove the following.

\begin{theorem}[Lift of $u$ is affine]\label{th:affine}
Assume that $(X,d,\fm)$ is an $\RCD(1,\infty)$-space,
and let $u\in D(\Delta)$ satisfy $\Delta u=-u$.
Then the function $\cU$ in \eqref{eq:U} is well-defined on $\cP^2(X)$,
and is affine on $\cP^2_{\ac}(X)$ in the sense that,
for every $L^2$-Wasserstein geodesic $(\mu_t)_{t \in [0,1]}$
with $\mu_0,\mu_1 \in \cP^2_{\ac}(X)$, we have
\[ \cU(\mu_t) =(1-t)\cU(\mu_0) +t\cU(\mu_1)
 \qquad \text{for all}\ t \in [0,1]. \]
\end{theorem}

Notice that \eqref{eq:Kcon} implies $\mu_t \in \cP^2_{\ac}(X)$ for all $t \in (0,1)$.
We will see that $u$ itself is affine in Proposition~\ref{pr:u-affine}.

Before beginning the proof of Theorem~\ref{th:affine},
let us recall some notations from the \emph{$\Gamma$-calculus}. 
Note that we will use some of these notations even before knowing 
that the underlying metric measure space satisfies
the $\CD (K,\infty)$ condition. 
We moreover will use the following notational convention:
Given a subspace $V \subset L^2(X)$, we define
\[ D_V(\Delta) :=\{ f \in D(\Delta) \,|\, \Delta f \in V \}. \]
{In addition, 
$W^{1,2,(\infty)} (X) : = \{ f \in W^{1,2} (X) \mid f, | \nabla f | \in L^\infty (X) \}$.}
For instance, we have
\[
\mathbb{D}_{\infty}(X) 
:=
D_{W^{1,2}}(\Delta) 
\cap 
{W^{1,2,(\infty)} (X)}.
% \left\{f\in W^{1,2}(X)\cap L^{\infty}(X) \mid |\nabla f|\in L^{\infty}(X)\right\}
\]
%by denoting
%\[
%W^{1,\infty} (X) : = \{ f \in W^{1,2} (X) \,|\, |\nabla f| \in L^\infty (X) \}.
%\]
{Recall 
% that $f \in D (\Delta) \cap L^\infty (X)$ implies $| \nabla f |^2 \in L^2 (X)$ and
that $f \in D_{L^\infty} (\Delta) \cap L^\infty (X)$ implies $| \nabla f | \in L^\infty (X)$ 
on $\RCD (K, \infty)$ spaces \cite[Theorem~3.1]{ams}.
In particular, 
$D_{L^\infty} (\Delta) \cap L^\infty (X) \subset D_{L^\infty} (\Delta) \cap W^{1,2,(\infty)} (X)$ 
holds on $\RCD (K,\infty)$ spaces.}
 
For $f \in D_{W^{1,2}}(\Delta)$ and $\phi \in D_{L^{\infty}}(\Delta)\cap L^{\infty}(X)$,
the \emph{$\Gamma_2$-operator} is defined as the integral of the left-hand side
of the Bochner inequality \eqref{eq:Boch} or \eqref{eq:Sav} against the test function $\phi$:
\[ \Gamma_2(f;\phi) :=\frac{1}{2}\int_X |\nabla f|^2 \Delta \phi \,d\fm
 -\int_X \langle\nabla f,\nabla \Delta f \rangle\phi \,d\fm. \]
Then the weak form of the Bochner inequality \eqref{eq:Boch}
(also called the \emph{$\Gamma_2$-inequality}) is written as,
provided that $\phi \ge 0$,
\begin{equation} \label{eq:Boch1}
 \Gamma_2(f;\phi) \ge \int_X |\nabla f|^2 \phi \,d\fm. 
\end{equation}
Similarly, the improved Bochner inequality \eqref{eq:Sav} is written as 
\begin{equation} \label{eq:Boch2}
 \Gamma_2(f;\phi) \ge \int_X | \Hess f |_{\HS}^2 \phi \, d \fm + \int_X |\nabla f|^2 \phi \,d\fm 
\end{equation}
for $f \in \mathbb{D}_\infty (X)$ and $\phi \ge 0$ \cite[Theorem~3.3.8]{Gi-dg}. 
We in addition define for later use
\begin{equation} \label{eq:Gamma2'} 
\Gamma'_2(f;\phi) :=\frac{1}{2} \int_X |\nabla f|^2 \Delta \phi \,d\fm
 +\int_X (\Delta f)^2 \phi \,d\fm +\int_X \langle\nabla \phi,\nabla f\rangle\Delta f \,d\fm 
\end{equation}
for $f \in D(\Delta)$ and $\phi \in D_{L^{\infty}}(\Delta)\cap {W^{1,2,(\infty)} (X)}$. 
In the intersection of the domains of $\Gamma_2$ and of $\Gamma'_2$ -- that is,
for $f \in D_{W^{1,2}}(\Delta)$ and 
$\phi \in D_{L^{\infty}}(\Delta) \cap {W^{1,2,(\infty)} (X)}$ --   
the integration by parts shows that $\Gamma_2(f;\phi)$ and $\Gamma'_2(f;\phi)$ coincide. 

We collect some properties of $u$ derived from the Bochner inequality
in the next proposition.

\begin{proposition}[Properties of $u$]\label{pr:eigen}
Let $(X,d,\fm)$ be an $\RCD(1,\infty)$-space, and consider $u \in D(\Delta)$ with $\Delta u=-u$.
Then,
\begin{enumerate}[{\rm (i)}]
\item $\Hess u=0$ holds $\fm$-almost everywhere$;$
\item $|\nabla u|$ is constant $\fm$-almost everywhere.
\end{enumerate}
\end{proposition}

\begin{proof}
We first show that 
{we can replace $\Gamma_2$ with $\Gamma_2'$ 
in the improved Bochner inequality \eqref{eq:Boch2} for $f \in D(\Delta)$
and nonnegative $\phi \in D_{L^\infty}(X) \cap L^\infty (X)$} 
by slightly modifying the discussion in \cite[Corollary~3.3.9]{Gi-dg}. 
Recall that $\Hess f$ is well-defined for $f\in D ( \Delta )$.
Pick a sequence $\{f_n\}_{n \in \N} {\subset} \bbD_{\infty}(X)$
such that $f_n, |\nabla f_n|, \Delta f_n$ converge to $f, |\nabla f|, \Delta f$ in $L^2(X)$, respectively.
Since $f_n \in \bbD_\infty (X) \subset D_{W^{1,2}} (\Delta)$, 
the improved Bochner inequality \eqref{eq:Boch2} 
together with the remark after the definition of $\Gamma_2'$ yields
\begin{equation*} 
\Gamma_2' (f_n ; \phi) 
\ge 
\int_X | \Hess f_n |_{\HS}^2 \phi \, d \fm + \int_X | \nabla f_n |^2 \phi \,d\fm .
\end{equation*}
Since $\phi, |\nabla \phi|, \Delta \phi \in L^{\infty}(X)$,
the hypotheses on $f_n,|\nabla f_n|,\Delta f_n$ imply the convergences of the left-hand side
and the second term in the right-hand side 
to the corresponding quantities for $f$, 
for instance,
\[ \left| \int_X (\Delta f_n)^2 {\phi} \,d\fm -\int_X (\Delta f)^2 {\phi} \,d\fm \right|
 \le \|{\phi}\|_{L^{\infty}} \|\Delta(f_n-f)\|_{L^2} \|\Delta(f_n+f)\|_{L^2}
 \to 0 \]
as $n \to \infty$.
Moreover, by \cite[Corollary~3.3.9]{Gi-dg}, $|\!\Hess (f_n-f)|_{\HS} \to 0$ in $L^2(X)$.
Hence we have, by {taking the limit}, 
\begin{equation} \label{eq:Boch3}
\Gamma_2' (f ; \phi) 
\ge 
\int_X | \Hess f |_{\HS}^2 \phi \, d \fm + \int_X | \nabla f |^2 \phi \,d\fm .
\end{equation}
This is nothing but what we claimed. 

(i)
Applying the improved Bochner inequality \eqref{eq:Boch3} to $f = u$ and $\phi \equiv 1$
(recall $\fm(X)<\infty$ from Lemma~\ref{lm:CD}(ii))
and using the integration by parts, we have
\[ 0 \ge \int_X |\!\Hess u|_{\HS}^2 \,d\fm. \]
Therefore $\Hess u=0$ holds $\fm$-almost everywhere.

(ii)
{Since $u \in D_{W^{1,2}} (\Delta)$ by $\Delta u = - u$,}
\eqref{eq:Boch1} yields
\begin{equation}\label{eq:const}
\int_X |\nabla u|^2 \Delta\phi \,d\fm \ge 0.
\end{equation}
for nonnegative $\phi \in D_{L^{\infty}}(\Delta) \cap L^{\infty}(X)$. 
One indeed has equality by replacing $\phi$ with $\|\phi\|_{L^{\infty}}-\phi$.

Now, thanks to the log-Sobolev inequality following from the $\RCD(1,\infty)$-condition,
we have the hypercontractivity of $\sH_t$.
Therefore $|\nabla u| \in L^2(X)$ yields $\sH_t(|\nabla u|) \in L^4(X)$
for sufficiently large $t>0$.
Combining this with $\sH_t u=\e^{-t} u$ (since $\Delta u=-u$)
and the gradient estimate $|\nabla (\sH_t u)| \le \e^{-t} \sH_t(|\nabla u|)$ (see \cite{Sa}),
we obtain $|\nabla u| \in L^4(X)$ and hence $|\nabla u|^2 \in L^2(X)$.
Then we deduce from equality in \eqref{eq:const} that
\begin{align*}
0 &=\int_X |\nabla u|^2 \Delta (\sH_t \phi) \,d\fm
 =\int_X |\nabla u|^2 \sH_t(\Delta \phi) \,d\fm
 =\int_X \sH_t(|\nabla u|^2) \Delta\phi \,d\fm \\
&=\int_X \Delta[\sH_t(|\nabla u|^2)] \phi \,d\fm.
\end{align*}
Since $\phi$ was arbitrary, $\Delta [\sH_t(|\nabla u|^2)]=0$ holds $\fm$-almost everywhere.
The Poincar\'e inequality then implies that  $\sH_t(|\nabla u|^2)$ is constant
$\fm$-almost everywhere.
Finally, letting $t \to 0$, we conclude that $|\nabla u|^2$ is constant $\fm$-almost everywhere.

Using the framework provided in \cite{Gi-dg} there is also an alternative way to argue for deducing the claim that $|\nabla u|^2$ is constant. 
Proposition 3.3.22 (ii) in \cite{Gi-dg} ensures that $|\nabla u|^2$ belongs to $H^{1,1}(X)$ with $d|\nabla u|^2=0$. $H^{1,1}(X)$ is the closure of $\bbD_{\infty}(X)$ 
in $W^{1,1}(X)$ and $d$ denotes the exterior derivative on a metric measure space.
Then Proposition 3.3.14 (ii) in \cite{Gi-dg} yields that $|\nabla u|^2$ also belongs to the Sobolev class $\mathcal{S}^2(X)$ with the same differential. 
Finally the Sobolev-to-Lipschitz property applies and yields the claim.
%
%Now, let $\sH^1_t$ be the $L^1$-extension of $\sH_t$
%that is a strongly continuous contraction semigroup on $L^1(X)$,
%and let $\Delta^1$ be its generator that is the $L^1$-extension of $\Delta$ 
%(see \cite[Proposition~2.4.2]{BH} or \cite[\S 2.1]{AGS-be}).
%For arbitrary $\phi \in D_{L^{\infty}}(\Delta)\cap L^{\infty}(X)$,
%we have $\sH_t\phi \in\mathbb{D}_{\infty}(X)$ and
%\begin{align*}
%0 &=\int_X |\nabla u|^2 \Delta (\sH_t \phi) \,d\fm
% =\int_X |\nabla u|^2 \sH_t(\Delta \phi) \,d\fm
% =\int_X \sH^1_t(|\nabla u|^2) \Delta\phi \,d\fm \\
%&=\int_X \Delta^1[\sH^1_t(|\nabla u|^2)] \phi \,d\fm.
%\end{align*}
%The last equality follows from the fact (see again \cite[Proposition~2.4.2]{BH}):
%\[ \int_X (\sH_t^1 g) f \,d\fm =\int_X g \sH_t f \,d\fm
% \quad \text{for all}\ g\in L^1(X),\ f \in L^{\infty}(X). \]
%Since $\phi$ was arbitrary, $\Delta^1[\sH^1_t(|\nabla u|^2)]=0$ holds $\fm$-almost everywhere.
%Finally, $\sH^1_t(|\nabla u|^2)$ converges to $|\nabla u|^2$ in $L^1(X)$ as $t \to 0$
%since $\sH_t^1$ is strongly continuous.
%Therefore $|\nabla u|^2$ is constant $\fm$-almost everywhere.
$\qedd$
\end{proof}

Since $u$ is not constant, we can normalize $u$ so as to satisfy
\begin{equation}\label{eq:|du|=1}
|\nabla u|=1 \qquad \text{$\fm$-almost everywhere}.
\end{equation}
Hence $u$ is $1$-Lipschitz.
This in particular implies that $u$ has at most linear growth,
therefore $\cU(\mu)$ in Theorem~\ref{th:affine} is well-defined.

Next we prove a key result concerning bounded functions,
generalizing the argument in \cite{Ke-obata} for $\RCD(K,N)$-spaces
to $\RCD(K,\infty)$-spaces.
This will be applied to approximations of the unbounded eigenfunction $u$.

\begin{theorem}[Hessian bound implies convexity]\label{th:kappa}
Let $(X,d,\fm)$ be an $\RCD(K,\infty)$-space with $K \in \R$ and $v \in \bbD_{\infty}(X)$
satisfy $\|v\|_{L^{\infty}} \le C<\infty$ and $\Hess v \ge -\kappa$ for some
$\kappa \in \R$.
Then the function $\cV(\mu):=\int_X v \,d\mu$ is $(-\kappa)$-convex on $\cP^2_{\ac}(X)$
in the sense that
\[ \cV(\mu_t) \le (1-t)\cV(\mu_0) +t\cV(\mu_1) +\frac{\kappa}{2}(1-t)t W_2^2(\mu_0,\mu_1) \]
for all $t \in [0,1]$ along any $W_2$-geodesic $(\mu_t)_{t \in [0,1]} \subset \cP^2(X)$
with $\mu_0,\mu_1 \in \cP^2_{\ac}(X)$.
\end{theorem}

The condition $\Hess v \ge -\kappa$ means that 
{$H[v](\phi,\phi) \ge -\kappa|\nabla \phi|^2$}
$\fm$-almost everywhere for all $\phi \in \bbD_{\infty}(X)$ 
(Recall the relation between $H[v]$ and $\Hess v$ reviewed in the last section).
In order to prove this theorem, similarly to \cite{Ke-obata,St-gf},
we introduce the modified measure
\begin{equation}\label{eq:tildeu}
\tilde{\fm} :=\e^{-v} \fm
\end{equation}
and consider the space $(X,d,\tilde{\fm})$.
We will denote by $\widetilde{\Delta}${, $\widetilde{\Gamma}_2$ 
and $\widetilde{\Gamma}_2'$
the Laplacian, the $\Gamma_2$-operator and the modified $\Gamma_2$-operator as  
in \eqref{eq:Gamma2'}} with respect to $\tilde{\fm}$.
We easily observe that, for $p \in [1,\infty]$,
\[ \e^{-C/p} \|f\|_{L^p({\fm})} \le \|f\|_{L^p(\tilde{\fm})} \le \e^{C/p} \|f\|_{L^p(\fm)} \]
for all $f \in L^p(X,\fm)=L^p(X,\tilde{\fm})$, and
\[ \e^{-C/p} \big\| |\nabla f| \big\|_{L^p({\fm})} \le \big\| |\nabla f| \big\|_{L^p(\tilde{\fm})}
 \le \e^{C/p} \big\| |\nabla f| \big\|_{L^p(\fm)} \]
for all {$f \in W^{1,2}(X,\fm)=W^{1,2}(X,\tilde{\fm})$. 
In addition, the minimal weak upper gradient of $f \in W^{1,2} (X, \tilde{\fm})$ 
induced by $\tilde{\fm}$ coincides with $|\nabla f|$} 
(see \cite[Lemma~4.11]{AGS-hf}). 
We moreover observe the following.

\begin{lemma}\label{importantlemma2}
Consider $v$ and $(X,d,\tilde{\fm})$ as above.
Then we have $D(\widetilde{\Delta})=D(\Delta)$ and, for any $f \in D(\widetilde{\Delta})$,
\begin{enumerate}[{\rm (i)}]
\item $\widetilde{\Delta}f =\Delta f- \langle \nabla v,\nabla f \rangle$,

\item $\|\widetilde{\Delta} f\|_{L^2(\tilde{\fm})}^2
 \le 2\e^{C/2} \left( \| \Delta f \|_{L^2({\fm})}^2
 +\| {| \nabla v |} \|_{L^{\infty}}^2 \big\| |\nabla f| \big\|_{L^2({\fm})}^2 \right)$,

\item $\| \Delta f \|_{L^2({\fm})}^2
 \le 2\e^{C/2} \left( \|\widetilde{\Delta} f\|_{L^2(\tilde{\fm})}^2
 +\| {| \nabla v |} \|_{L^{\infty}}^2 \big\| |\nabla f| \big\|_{L^2(\tilde{\fm})}^2 \right)$.
\end{enumerate}
In particular, if $f\in D(\widetilde{\Delta})$,
then $\sH_t f \in D(\widetilde{\Delta})$ and $\sH_t f \to f$ in $D(\widetilde{\Delta})$ as $t \to 0$.
\end{lemma}

\begin{proof}
Consider $f \in D(\widetilde{\Delta}) \subset W^{1,2}({X}, {\tilde\fm})$,
then $\widetilde{\Delta}f+\langle \nabla f,\nabla v\rangle\in L^2({X},\tilde{\fm})=L^2({X},\fm)$.
Given $g\in W^{1,2}({X},\fm)$, we have $\tilde{g}:=\e^v g\in W^{1,2}({X},\fm) =W^{1,2}({X},\tilde{\fm})$ and
\begin{align*}
\int_X \big( \widetilde{\Delta} f +\langle \nabla v,\nabla f \rangle \big) g \,d\fm
&= \int_X \tilde{g} \widetilde{\Delta} f \,d\tilde{\fm}+\int_X \langle \nabla v,\nabla f \rangle g \,d\fm\\
&= -\int_X \langle \nabla \tilde{g},\nabla f \rangle \,d\tilde{\fm}+\int_X \langle \nabla v,\nabla f \rangle g \,d\fm\\
 %= -\int_X \langle \nabla \tilde{g},\nabla f \rangle \e^{-v} \,d{\fm}+\int_X \langle \nabla v,\nabla f \rangle g \,d\fm\\
&= -\int_X \left\{\langle \nabla g,\nabla f\rangle +\langle \nabla v,\nabla f \rangle g\right\} \,d\fm+\int_X \langle \nabla v,\nabla f \rangle g \,d\fm\\
&= -\int_X \langle \nabla g,\nabla f \rangle \,d\fm.
\end{align*}
This shows $f \in D(\Delta)$ and the equation in (i).
Similarly, $f\in D(\Delta)$ implies $f\in D(\widetilde{\Delta})$ (hence $D(\Delta)=D(\widetilde{\Delta})$)
and the equation in (i).
(ii) and (iii) follow easily from (i).
$\qedd$
\end{proof}

Now, we pick 
$f \in \mathbb{D}_{\infty}(X)$ 
and $\phi \in D_{L^\infty} (\Delta) \cap W^{1,2,(\infty)} (X)$.
Then 
\begin{align}
\widetilde{\Gamma}'_2(f; \phi) \nonumber
&=\frac{1}{2}\int_X |\nabla f|^2 \widetilde{\Delta}\phi \,d\tilde{\fm}
 +\int_X (\widetilde{\Delta} f)^2 \phi \,d\tilde{\fm}
 +\int_X \langle \nabla \phi,\nabla f \rangle \widetilde{\Delta} f \,d\tilde{\fm}
 \nonumber\\
&=: \text{(I)} +\text{(II)} +\text{(III)}
 \label{eq:123}
\end{align}
is well-defined. 

\begin{proposition}\label{pr:tildeG}
Let $v$ and $(X,d,\tilde{\fm})$ be as in \eqref{eq:tildeu}.
Then, for $f \in \mathbb{D}_{\infty}(X)$ 
and $\phi \in D_{L^\infty} (\Delta) \cap W^{1,2,(\infty)} (X)$ 
with $\phi \ge 0$,  
we have 
\begin{equation} \label{eq:approximation0}
\widetilde{\Gamma}'_2(f;\phi)
\ge 
(K-\kappa) \int_X |\nabla f|^2 \phi \,d\tilde{\fm}. \
\end{equation}
\end{proposition}

\begin{proof}
Observe first that $\e^{-v} \in D(\Delta)$ from
\[ \Delta(\e^{-v}) =-\e^{-v} \Delta v +\e^{-v} |\nabla v|^2 \in L^2(X). \]
{Moreover $\e^{-v} \phi \in D(\Delta)$ 
and we have 
\[
\Delta ( \e^{-v} \phi ) 
= 
\phi \Delta ( \e^{-v} ) 
-2 \langle \nabla v , \nabla \phi \rangle \e^{-v} 
+ \e^{-v} \Delta \phi
\]
as expected (see \cite[Theorem~4.29]{Gi-Ondiff}).}
We shall compute (I), (II) and (III) in \eqref{eq:123} in order,
and then compare 
{%
$\widetilde{\Gamma}'_2(f;\phi)$ with $\Gamma'_2(f;\e^{-v} \phi)$.}
Let us begin with
\begin{align*}
&2\text{(I)}
 =\int_X |\nabla f|^2 \e^{-v} \Delta \phi \,d\fm
 -\int_X |\nabla f|^2 \langle \nabla \phi,\nabla v\rangle \e^{-v} \,d\fm \\
&= \int_X |\nabla f|^2 \Delta (\e^{-v} \phi) \,d\fm
 -\int_X |\nabla f|^2 (\Delta \e^{-v}) \phi \,d\fm
 +\int_X |\nabla f|^2 \langle \nabla \phi,\nabla v\rangle \e^{-v} \,d\fm \\
&= \int_X |\nabla f|^2 \Delta (\e^{-v} \phi )\,d\fm
 +\int_X \langle\nabla (|\nabla f|^2 \phi),\nabla \e^{-v} \rangle \,d\fm
 +\int_X |\nabla f|^2 \langle \nabla \phi,\nabla v \rangle \e^{-v} \,d\fm \\
&= \int_X |\nabla f|^2 \Delta (\e^{-v} \phi) \,d\fm
 -\int_X \langle \nabla (|\nabla f|^2),\nabla v \rangle \e^{-v} \phi \,d\fm.
\end{align*}
Here, the first equality follows from Lemma~\ref{importantlemma2}(i),
the second equality is the Leibniz rule for $\Delta$,
the third equality is the integration by parts,
and the fourth equality is the Leibniz rule for $\nabla$,
where we note again that $|\nabla f|^2 \in W^{1,2}(X)$ since $f \in \bbD_{\infty}(X)$.

Next we have, again by Lemma~\ref{importantlemma2}(i),
the integration by parts and the Leibniz rule for $\nabla$,
\begin{align*}
&\text{(II)}
 = \int_X (\Delta f)^2 \phi \,d\tilde{\fm}
 +\int_X \langle \nabla f,\nabla v\rangle^2 \phi \,d\tilde{\fm}
 -2\int_X \Delta f \langle \nabla f,\nabla v\rangle \phi \,d\tilde{\fm} \\
&= \int_X (\Delta f)^2 \phi \,d\tilde{\fm}
 +\int_X \langle \nabla f,\nabla v\rangle^2 \phi \,d\tilde{\fm}
 +2\int_X \big\langle \nabla f,
 \nabla \big( \langle \nabla f,\nabla v\rangle \e^{-v} \phi \big) \big\rangle \,d\fm \\
&= \int_X (\Delta f)^2 \phi \,d\tilde{\fm}
 -\int_X \langle \nabla f,\nabla v\rangle^2 \phi \,d\tilde{\fm} \\
&\quad +2\int_X \big\langle \nabla f,\nabla \langle \nabla f,\nabla v\rangle \big\rangle
 \e^{-v} \phi \,d\fm
 +2\int_X \langle \nabla f,\nabla \phi \rangle
 \langle \nabla f,\nabla v\rangle \e^{-v} \,d\fm.
\end{align*}
Finally,
\begin{align*}
&\text{(III)}
 = \int_X \langle\nabla \phi,\nabla f\rangle \e^{-v} \Delta f \,d\fm
 -\int_X \langle\nabla \phi,\nabla f\rangle
 \langle \nabla v,\nabla f \rangle \e^{-v} \,d\fm \\
&= \int_X \langle\nabla (\e^{-v} \phi),\nabla f \rangle \Delta f \,d\fm
 +\int_X \langle\nabla v,\nabla f\rangle \e^{-v} \phi \Delta f \,d\fm \\
&\quad - \int_X \langle\nabla \phi,\nabla f\rangle
 \langle \nabla v,\nabla f \rangle \e^{-v} \,d\fm \\
&= \int_X \langle\nabla (\e^{-v} \phi),\nabla f\rangle \Delta f \,d\fm
 -\int_X \big\langle\nabla \big(\langle\nabla v,\nabla f\rangle \e^{-v} \phi \big),
 \nabla f \big\rangle \,d\fm \\
&\quad -\int_X \langle\nabla \phi,\nabla f\rangle
 \langle \nabla v,\nabla f \rangle \e^{-v} \,d\fm \\
&= \int_X \langle\nabla (\e^{-v} \phi),\nabla f\rangle \Delta f \,d\fm
 -\int_X \big\langle \nabla \langle \nabla v,\nabla f \rangle,\nabla f \big\rangle \e^{-v} \phi \,d\fm \\
&\quad +\int_X \langle\nabla v,\nabla f \rangle^2 \e^{-v} \phi \,d\fm
 -2\int_X \langle\nabla \phi,\nabla f\rangle
 \langle \nabla v,\nabla f \rangle \e^{-v} \,d\fm.
\end{align*}

Adding (I), (II) and (III) yields
\begin{align*}
\widetilde{\Gamma}'_2(f;\phi)
&= \Gamma'_2(f;\e^{-v} \phi)
 -\frac{1}{2} \int_X \langle\nabla(|\nabla f|^2),\nabla v\rangle \phi \,d\tilde{\fm}
 +\int_X \big\langle \nabla f,\nabla \langle \nabla f,\nabla v \rangle \big\rangle \phi \,d\tilde{\fm} \\
&= \Gamma'_2(f;\e^{-v} \phi)
 +\int_X {H[v](f,f)} \phi \,d\tilde{\fm}.
\end{align*}
Notice that $\Gamma_2'(f;\e^{-v} \phi)$ is well-defined
for $f \in \bbD_\infty (X)$ and $\e^{-v} \phi \in D(\Delta) \cap L^{\infty}(X)$ 
since $|\nabla f|, | \nabla \phi | \in L^\infty (X)$, and we have
\[ \Gamma_2'(f;\e^{-v} \phi)
 \ge K \int_X |\nabla f|^2 \e^{-v} \phi \,d\fm
 =K\int_X |\nabla f|^2 \phi \,d\tilde{\fm} \]
by $\RCD(K,\infty)$ {condition} (see \cite[Corollary~4.3]{ams}).
Now we apply the hypothesis $\Hess v \ge -\kappa$ to conclude
\[ \widetilde{\Gamma}'_2(f;\phi) \ge (K-\kappa) \int_X |\nabla f|^2 \phi \,d\tilde{\fm} \]
as desired.
$\qedd$
\end{proof}

We shall extend the class of functions $f$ and $\phi$ 
in the last proposition.
For this purpose, we introduce a mollification $\mathfrak{h}_\varepsilon$ 
given by 
\[
\mathfrak{h}_\varepsilon f 
:= 
\int_0^\infty 
  \frac{1}{\varepsilon} 
  \eta \left( \frac{t}{\varepsilon} \right)
  \sH_t f 
\, d t , 
\]
where $\eta \in C^\infty_c ( (0 , \infty ))$ with $\eta \ge 0$ 
and $\int_0^\infty \eta (t) \, d t = 1$ and $f \in L^2 (X)$. 
We can easily see that $\mathfrak{h}_\varepsilon f \in \bbD_\infty (X)$ 
and moreover $\mathfrak{h}_\varepsilon f \in D_{L^\infty} (\Delta)$ 
if $f \in L^2 (X) \cap L^\infty (X)$. 
As $\varepsilon \to 0$, 
$\mathfrak{h}_\varepsilon f \to f$ occurs 
both in $W^{1,2} (X)$ and in $D (\Delta)$. 
We also consider another mollification $\tilde{\mathfrak{h}}_\varepsilon$ 
by using the heat semigroup $( \tilde{\sH}_t )_{t \ge 0}$ 
associated with $\tilde{\Delta}$ on $L^2 (\tilde{\fm})$ 
instead of $( \sH_t )_{t \ge 0}$.

% \begin{lemma}\label{lm:t->0}
% For $f \in D(\widetilde{\Delta}) \cap L^{\infty}(X)$ 
% and $\phi \in L^{\infty}(X)$ with $\phi \ge 0$, 
% we have for any $s>0$
% \[ \widetilde{\Gamma}'_2(f;\sH_s\phi)
%  \ge (K-\kappa) \int_X |\nabla f|^2 \sH_s\phi \,d\tilde{\fm}. \]
% \end{lemma}

% \begin{proof}
% Notice that $\widetilde{\Gamma}'_2(f;\sH_s\phi)$ is well-defined.
% \textcolor{magenta}{Need\cite[Theorem~3.1]{ams}}
% Since $\sH_s \phi \in D_{L^{\infty}}(\widetilde{\Delta}) \cap L^{\infty}(X)$
% \textcolor{magenta}{Need mollification of $\sH_s$}
% and $\sH_tf \to f$ in $D(\widetilde{\Delta})$ by Lemma~\ref{importantlemma2} 
% and $\sH_s \phi \in \textcolor{magenta}{W^{1,\infty} (X)}$ by Proposition~\ref{pr:Lip}, 
% it follows that
% \[ \widetilde{\Gamma}'_2(\sH_tf;\sH_s\phi) \to \tilde{\Gamma}'_2(f;\sH_s\phi),
%  \qquad |\nabla \sH_tf| \overset{L^2(X)}{\to} |\nabla f| \]
% as $t \to 0$.
% Combining these with Proposition~\ref{pr:tildeG} yields the claim.
% $\qedd$
% \end{proof}

\begin{proposition}\label{pr:s->0}
$(X,d,\tilde{\fm})$ satisfies $\RCD(K-\kappa,\infty)$.
\end{proposition}

\begin{proof}
First, since $(X,d,\fm)$ satisfies the condition $\RCD(K,\infty)$, we have that $$\{f\in W^{1,2}(X)\mid |\nabla f|\leq C\}\subset \{f\in \mbox{Lip}(X,d)\mid |\nabla^L f|\leq C\}$$
by the Sobolev-to-Lipschitz property. Moreover, $W^{1,2}(X,\fm)=W^{1,2}(X,\tilde{\fm})$, $\tilde{\fm}$ is locally finite, 
and one checks that $(X,d,\tilde{\fm})$ again satisfies an exponential growth condition. More precisely, the latter means that there exist constants
$M>0$ and $c>0$ such that 
$$\tilde{\fm}(B_r(x))\leq M\exp(c r^2)\ \mbox{ for every }r>0.$$
Hence, we can apply Corollary 4.18 (ii) in \cite{AGS-be} and consequently
it suffices to show the Bochner inequality
\begin{equation}\label{eq:tildeBE}
\widetilde{\Gamma}_2(f;\phi)
 \ge (K-\kappa) \int_X |\nabla f|^2 \phi \,d\tilde{\fm}
\end{equation}
for any $f \in D_{W^{1,2}}(\widetilde{\Delta})$
and $\phi \in D_{L^{\infty}}(\widetilde{\Delta}) \cap L^{\infty}(X)$ with $\phi \ge 0$.

Let us first assume in addition $f \in L^\infty (X)$ 
and $\phi \in W^{1,2,(\infty)} (X)$.
We remark that we do not know whether $\Delta f \in W^{1,2} (X)$ or not, 
while the regularization property as Proposition~\ref{pr:Lip} is available 
for $\mathfrak{h}_\varepsilon$ but not for $\tilde{\mathfrak{h}}_\varepsilon$. 
Let $f_{\varepsilon} : = \mathfrak{h}_\varepsilon f$ for $\varepsilon > 0$. 
Then $f_{\varepsilon} \in \bbD_{\infty} (X)$ and \eqref{eq:approximation0} 
holds for $f_{\varepsilon}$ and $\phi$ by Proposition~\ref{pr:tildeG}. 
Since $f_\varepsilon \to f$ in $D(\Delta)$, we have 
$\| \tilde{\Delta} f_{\varepsilon} - \tilde{\Delta} f \|_{L^2 (X,\tilde\fm)} \to 0$ 
as $\varepsilon \to 0$. 
Thus,  
we obtain \eqref{eq:approximation0} for $f$ and $\phi$
by letting $\varepsilon \to 0$. 
Since $\tilde{\Gamma}_2 ( f;  \phi ) = \tilde{\Gamma}_2' ( f;  \phi )$
under our assumption on $\phi$, the assertion \eqref{eq:tildeBE} holds. 

Next we drop the assumption 
$| \nabla \phi | \in L^\infty (X)$ in the first step. 
Since $\sH_s \phi \in W^{1,2,(\infty)} (X)$ for $s > 0$ 
by virtue of Propostion~\ref{pr:Lip}, 
\eqref{eq:tildeBE} holds for $f$ and $\sH_s \phi$.
We will let $s \to 0$ in this inequality. 
Note that \cite[Theorem~3.1]{ams} yields $| \nabla f | \in L^4 (X)$. 
Since $\sH_s \phi \to \phi$ in $D(\Delta)$ as $s \to 0$, 
we have 
$\| \tilde{\Delta} \sH_s \phi - \tilde{\Delta} \phi \|_{L^2 (X,\tilde{\fm})} \to 0$. 
Moreover, we have $\sH_s \phi \to \phi$ 
with respect to the weak-$\star$-topology in $L^\infty (X)$. 
Thus, by virtue of 
$| \nabla f | \in L^2 (X) \cap L^4 (X)$ and 
$| \nabla \tilde{\Delta} f | \in L^2 (X)$,
we obtain \eqref{eq:tildeBE} 
for our choice of $f$ and $\phi$ by letting $s \to 0$. 
Finally we will show \eqref{eq:tildeBE} for $f \in D_{W^{1,2}} (\tilde{\Delta})$ 
and nonnegative $\phi \in D_{L^\infty} (\tilde{\Delta}) \cap L^\infty (X)$. 
For $R > 0$ and $\varepsilon > 0$, let 
$f_{\varepsilon , R} : = \tilde{\mathfrak{h}}_{\varepsilon} ( (-R) \vee f \wedge R$). 
Then $f_{\varepsilon, R} \in D_{W^{1,2}} (\tilde{\Delta}) \cap L^\infty (X)$ 
and \eqref{eq:tildeBE} holds for $f_{\varepsilon, R}$ and $\phi$. 
Then, letting $R \to \infty$ and $\varepsilon \to 0$ afterwards, 
we obtain \eqref{eq:tildeBE} for $f$ and $\phi$ 
by arguing as in the proof of \cite[Theorem~4.8]{EKS} 
(see also \cite[Theorem~4.6]{GKO}, \cite[Corollary~2.3]{AGS-be}).
$\qedd$
\end{proof}

\noindent
{\it Proof of Theorem~$\ref{th:kappa}$}.
Let $\alpha>0$ and consider the scaled space $X_{\alpha}=(X,\alpha^{-1}d)$,
$v_{\alpha}=v/\alpha^2$ and $\tilde{\fm}_{\alpha} :=\e^{-v_{\alpha}}\fm$.
By definition we find $\Hess v_{\alpha} \ge -\kappa$ on $X_{\alpha}$.
It follows from Proposition~\ref{pr:s->0} that
\[ \Ent_{\tilde{\fm}_{\alpha}}(\mu) =\Ent_{\fm}(\mu) +\int_X v_{\alpha} \,d\mu
 =\Ent_{\fm}(\mu) +\frac{\cV(\mu)}{\alpha^2} \]
is strongly $(\alpha^2 K-\kappa)$-convex on $\cP^2(X_{\alpha})$
(recall Lemma~\ref{lm:CD}(i)).
Therefore, for any $L^2$-Wasserstein geodesic $(\mu_t)_{t \in [0,1]}$ over $X_{\alpha}$
with bounded supports and densities, we have
\[ \Ent_{\tilde{\fm}_{\alpha}}(\mu_t)
 \le (1-t)\Ent_{\tilde{\fm}_{\alpha}}(\mu_0) +t\Ent_{\tilde{\fm}_{\alpha}}(\mu_1)
 +\frac{\kappa-\alpha^2 K}{2} (1-t)t \frac{W_2(\mu_0,\mu_1)^2}{\alpha^2}, \]
where $W_2$ is with respect to $d$.
The boundedness ensures the finiteness of the entropies.
Multiplying this inequality with $\alpha^2$ and letting $\alpha \to 0$, 
we see that $\cV$ is $(-\kappa)$-convex along $(\mu_t)_{t \in [0,1]}$.
Then by the truncation and cut-off arguments, 
we can show the desired claim for arbitrary $W_2$-geodesics
$(\mu_t)_{t \in [0,1]}$ with $\mu_0,\mu_1 \ll \fm$ 
(for instance, follow the proof of \cite[Theorem~4.20]{StI}).
% Indeed, if we approximate $\mu_0$ and $\mu_1$ 
% by those with bounded supports and densities, the same 
% property holds for the $L^2$-Wasserstein geodesic joining 
% them by \cite[Theorem~4.2]{AGMR}). 
% Thus we can apply our argument for each of them. 
% Since midpoints have bounded entropies, 
% it has a weakly convergent subsequence. 
% Then the lower semicontinuity of $W_2$ with respect to 
% the weak convergence implies that any subsequential 
% limit is the midpoint of $\mu_0$ and $\mu_1$. 
% Then the conclusion follows from 
% the lower semi-continuity of $\Ent$ with respect to 
% weak convergece with uniformly bounded second moments 
% and \cite[Remark~4.6 (i)]{StI}. 
$\qedd$
\bigskip

%That is, for any Wasserstein geodesic $(\mu_t)_{t\in [0,1]}$ between $\fm$-absolutely continuous probability measures, one has
%\begin{align*}
%\int v_n \circ e_t d\Pi(\gamma)\geq \int \left[(1-t)v_n\circ e_0+t v_n\circ e_1+\textstyle{\frac{1}{2}}\kappa_nt(1-t)|\dot{\gamma}|^2\right] d\Pi(\gamma). 
%\end{align*}
%Since $(X,d,\fm)$ satisfies the condition $RCD(0,\infty)$, it is essentially non-branching.
%\newpage
%\begin{theorem}\label{secondtheorem}
%Assume $(X,d,\fm)$ satisfies the condition $RCD(1,\infty)$ for $K>0$, and let $u\in D_{L^2(\fm)}(\Delta)$ be an eigenfunction such that $\Delta u=-u$. Then 
%$|\nabla u|\in W^{1,2}_{loc}(\fm)$, the Hessian $\Hess u$ of $u$ is well-defined and measurable, and it satifies $\Hess u=0$.
%\end{theorem}
%\begin{proof}
%Since $u\in W^{1,2}(\fm)$ and since $g_n$ is smooth and bounded, by the chain rule for the minimal weak upper gradient 
%we can compute 
%\begin{align*}
%|\nabla v_n|=g_n'(u)|\nabla u|
%\end{align*}
%$v\in W^{1,\infty}(\fm)$. Moreover, by the chain rule for 
%generalized Laplacian it follows that $v_n\in D_{L^2(\fm)}(\Delta)$ and 
%\begin{align*}
%\Delta v_n= g'_n(u)\Delta u+g_n''(u)|\nabla u|^2=-g_n'(u)u+g_n''(u).
%\end{align*}
%Note, that $g_n$ is smooth and $g_n'''$ is bounded. Therefore, again by the chain rule for the minimal weak upper gradient, it follows $\Delta v_n\in W^{1,2}(\fm)$.
%Hence, $v_n\in \mathbb{D}_{\infty}(X)\cap L^{\infty}$.
%\smallskip\\
%\end{proof}

We are ready to prove Theorem~\ref{th:affine}.
Recall from \eqref{eq:|du|=1} that we can normalize $u$
so that $|\nabla u|=1$ $\fm$-almost everywhere,
and then $u$ is $1$-Lipschitz and $\cU$ is well-defined.
%Then, we have for some $x_0\in X$
%\[ u(x)\leq u(x_0)+d(x_0,x) \]
%and therefore, for any constant $C>0$ one finds $R>0$ such that 
%\[ C|u(x)|\leq C|u(x_0)|+d(x_0,x)^2\ \mbox{ for each $x\in X$ with }d(x,x_0)>R. \]
%Hence, for any constant $C>0$ one finds a constant $K>0$ such that
%\[ \int e^{C|u|}d\fm\leq K\int e^{d(\cdot,x_0)^2}d\fm<\infty. \]
\bigskip

\noindent
{\it Proof of Theorem~$\ref{th:affine}$.} 
In order to apply Theorem~\ref{th:kappa},
we smoothly truncate $u$ by using the function $g_n:\R \lra (-n\pi/2,n\pi/2)$
given by $g_n(r)=n\arctan(r/n)$ for $n \in \N$.
Note that $g_n(r) \to r$ as $n \to \infty$ uniformly on compact sets,
and we have
\[ g'_n(r) =\frac{1}{(r/n)^2+1} \in (0,1], \qquad
 g''_n(r) =-\frac{2{n^2}r}{(r^2+n^2)^2}. \]
Set 
$\kappa_n 
:=
\sup |g_n''|
=
% (2n/\sqrt{3})^{-3}
{3\sqrt{3}/(8n)}
$ which goes to $0$ as $n \to \infty$.

Define $v_n :=g_n\circ u$.
Then clearly $v_n \in L^{\infty}(X)$ as well as
$|\nabla v_n|=g_n'(u) \in L^{\infty}(X)$ by $|\nabla u|=1$.
Moreover, we have
\[ \Delta v_n =g_n'(u) \Delta u +g_n''(u)|\nabla u|^2
 =-g_n'(u)u +g_n''(u) \in W^{1,2} ({X}) \cap L^{\infty}(X). \]
Thereby $v_n \in \bbD_{\infty}(X)$ with $\Delta v_n \in L^{\infty}(X)$. 
In particular, we can apply the chain rule for the Hessian 
(Proposition 3.3.21 in \cite{Gi-dg}) yielding that 
\begin{multline*}
{H[v_n](\phi,\phi) =} \Hess v_n({\nabla \phi,\nabla \phi})
 =g_n'(u) \Hess u({\nabla \phi,\nabla \phi}) +g_n''(u) \langle \nabla u,\nabla \phi \rangle^2
\\
 =g_n''(u) \langle \nabla u,\nabla \phi\rangle^2
 \ge -\kappa_n |\nabla \phi|^2 
\end{multline*}
$\fm$-almost everywhere.
Hence we can apply Theorem~\ref{th:kappa} to $v_n$ and find that,
for every $W_2$-geodesic $(\mu_t)_{t \in [0,1]}$
with $\mu_0,\mu_1 \in \cP^2_{\ac}(X)$,
\begin{equation}\label{eq:kappaconvexity}
\int_X v_n \,d\mu_t \le (1-t) \int_X v_n \,d\mu_0 +t\int_X v_n \,d\mu_1
 +\frac{\kappa_n}{2} (1-t)t W_2^2(\mu_0,\mu_1).
\end{equation}
Now, we consider $(\mu_t)_{t \in [0,1]}$ such that $\mu_0$ and $\mu_1$ have bounded support.
Let $B \subset X$ be a bounded set with $\supp \mu_t \subset B$ for every $t \in [0,1]$.
By the $1$-Lipschitz continuity of $u$ and the uniform convergence of $g_n(r)$ to $r$ on bounded subsets in $\R$,
we obtain that $v_n \to u$ uniformly on $B$.
Hence taking the limit as $n \to \infty$ in \eqref{eq:kappaconvexity} yields
\[ \int_X u \,d\mu_t \le (1-t)\int_X u \,d\mu_0 +t\int_X u \,d\mu_1. \]
We can repeat the same argument for $-u$ in place of $u$, thus equality holds.
By an exhaustion of $X$ with bounded sets,
we conclude that $\cU$ is affine on whole $\cP^2_{\ac}(X)$.
%This yields that for every pair of points $x_0,x_1$
%there exists a geodesic $\gamma$ between them, such that $u\circ\gamma$ is affine.
$\qedd$

\section{Second step: Regular Lagrangian gradient flow of the eigenfunction}\label{sc:gf}%%%%%
%%%%%%%%%%%%%%%%%%%%%%%%

The eigenfunction $u$ as in the previous section will play the key role
in the same way that the Busemann function did in the splitting theorem
of $\RCD(0,N)$-spaces in \cite{Gi-split} (see also an overview \cite{Gi-ov}).
In order to overcome technical difficulties arising due to the lack of the volume doubling property,
we employ the regular Lagrangian flow of the negative gradient vector field $-\nabla u$,
and use it to construct and analyze the gradient flows of $\cU$ and then of $u$.

\subsection{Regular Lagrangian flow}\label{ssc:RLF}%%%%%
%%%%%%%%%%%%%%%

We apply the theory of \emph{regular Lagrangian flows}
developed by Ambrosio--Trevisan \cite{AT}
(as a far reaching generalization of the celebrated DiPerna--Lions theory \cite{DL},
see also the lecture notes \cite{AT2}) to the vector field $-\nabla u$,
where $u$ is the eigenfunction as in the previous section.
The notion of regular Lagrangian flow is closely related with 
the continuity equation. We begin with solving the continuity 
equation of $- \nabla u$ starting from $\fm$. 

\begin{proposition}[Solution to the continuity equation]\label{pr:meas}
% For $t \in \R$, let $\rho_t : =\e^{-tu-t^2/2}$. 
% Then $\rho_t$ solves the continuity equation for $- \nabla u$. 
A solution to the continuity equation for $- \nabla u$ is 
given by $\e^{- t u - t^2 / 2}$. 
That is, for any $f \in W^{1,2} (X)$ the map $t\mapsto \int e^{-tu-t^2/2}d\fm$ is absolutely continuous and its derivative is given for a.e. $t\in \mathbb{R}$ by
\[
\frac{d}{dt} \int_X f \e^{-tu - t^2/2} \,d \fm 
= 
-\int_X \langle \nabla f, \nabla u \rangle \e^{- t u - t^2 /2 } \, d \fm.
\]
\begin{proof}
Recalling $u=-\Delta u$ and $|\nabla u|=1$ from \eqref{eq:|du|=1}, we have
\begin{align*}
\frac{d}{dt} \int_X f \e^{- t u - t^2 /2 } \, d \fm
&=
\int_X f\e^{-tu-t^2/2}(-u-t) \,d\fm
 =
\int_X f\e^{-tu-t^2/2}(\Delta u-t) \,d\fm\\
&=
\int_X \e^{-tu-t^2/2} 
( -tf-\langle \nabla f,\nabla u \rangle +f \langle t\nabla u,\nabla u \rangle )
\,d\fm
\\
&=
-\int_X \langle \nabla f,\nabla u \rangle \e^{- t u - t^2 /2 } \, d \fm .
\end{align*}
$\qedd$
\end{proof}
\end{proposition}

\begin{theorem}[Regular Lagrangian flow of $-\nabla u$]\label{th:RLF}
There exists a unique map $($up to equality almost everywhere$)$
$F:X \times \R \lra X$ such that
\begin{enumerate}[{\rm (i)}]
\item 
$(F_t)_*\fm \le C(\cdot,t) \fm$ 
for a locally bounded function $C : X \times \R \lra ( 0 , \infty );$ 

\item 
$F_0$ is the identity map and, for every $x \in X$,
$t \longmapsto F_t(x)$ is a $1$-Lipschitz curve$;$

\item 
For every $f \in W^{1,2}(X)$ and $\fm$-almost every $x$,
the map $t \longmapsto f(F_t(x))$ is in $W^{1,2}_{\loc}(\R)$
and its distributional derivative satisfies
\[ \frac{d}{dt} \big[ f\big( F_t(x) \big) \big]
 =-\langle \nabla f,\nabla u \rangle \big( F_t(x) \big); \]

\item 
For each $s \in \R$, we have $F_t \circ F_s=F_{t+s}$ for every $t \in \R$ 
$\fm$-almost everywhere$;$

\item 
For $\fm$-almost every $x$,
the metric speed of the curve $t \longmapsto F_t(x)$ is 
constant and equal to $1$.
\end{enumerate}
Moreover, $( F_t )_* \fm = \e^{-tu - t^2/2} \fm$ holds. 
Note that equality almost everywhere is understood at the level of curves. More precisely, if there is another map $\tilde{F}$ as above then for $\fm$-a.e. $x$ we have that $F_t(x)=\tilde{F}_t(x)$ for every 
$t\in \R$.
\end{theorem}

Regular Lagrangian flows 
for gradient vector fields on $\RCD (K,\infty)$-spaces are 
studied in \cite[Theorems~9.7]{AT}.
However, the gradient vector field $\nabla u$ does not meet the assumption since $\Delta u$ is only in $L^{\infty}_{loc}(X)$ when $L^{\infty}(X)$ is required. 
% Indeed our condition (i) on $F$ in the last theorem 
% is weaker than the corresponding one in \cite[Definition~8.1]{AT}, 
% and the original one does not hold in our case. 
Therefore, - though we cannot exactly apply the results in \cite{AT} - 
we follow closely an argument of a more general result \cite[Theorem~8.3]{AT} 
to prove Theorem~\ref{th:RLF}, with the aid of Proposition~\ref{pr:meas}. 

In the proof of Theorem~\ref{th:RLF}, 
we freely use notions introduced in \cite{AT}.

\begin{proof}
We first localize the argument. 
Let $T > 0$ and fix $x_0 \in X$. 
Take $R > 3T$ and a Lipschitz cut-off function 
$\psi_R:X \lra [0,1]$ such that $\psi_R=1$ on $B_R(x_0)$,
$\supp \psi_R \subset B_{2R}(x_0)$ and $|\nabla \psi_R | \le R^{-1}$.
Then we consider the (autonomous) \emph{derivation}
$\bm{b}_R:=-\psi_R \cdot \nabla u$ 
(\cite[Definition~3.1]{AT}), 
namely
\[ 
\bm{b}_R:\mathscr{A} \ni f \,\longmapsto\,
 -\psi_R \cdot \langle \nabla f,\nabla u \rangle \in L^{\infty}(X), 
\]
where $\mathscr{A}$ is the set of Lipschitz functions on $X$ with bounded support.
Notice that $\mathscr{A}$ is dense in $W^{1,2}(X)$ 
(see \cite{AGS-rcd} for instance).

We claim the uniqueness of weak solutions to the \emph{continuity equation}
\begin{equation}\label{eq:conteq}
\frac{dv_t}{dt} +\div(v_t \cdot \bm{b}_R)=0
\end{equation}
for $\bm{b}_R$ with the initial condition $v_0 = \bar{v} \in L^2(X)$ 
(\cite[Definition~4.2]{AT})
in the class
\[ \mathcal{L}_+ :=\{ v \in L_t^{\infty}(L_x^{\infty}) \,|\,
 t \longmapsto v_t\ \text{is weakly continuous} \} \]
(notice that $L^1 (X) \cap L^{\infty} (X) = L^{\infty} (X)$
since $\fm(X)<\infty$).
This claim follows from \cite[Theorem~5.4]{AT} with $r=s=2$ and $q=\infty$
(we in fact have the uniqueness in the larger class $L_t^{\infty}(L_x^2)$).
Indeed, the hypotheses of the theorem are verified in our case as follows. 
We can easily construct a class of functions satisfying \cite[(4-3)]{AT}, 
and the $L^2$-$\Gamma$-inequality always holds 
(as mentioned after \cite[Definition~5.1]{AT}). 
As for the assumptions on $\bm{b}_R$, clearly $\bm{b}_R \in L^{\infty}(X)$ holds
in the sense that $|\psi_R \cdot \nabla u| \in L^{\infty}(X)$. 
By the definition of \emph{divergence} $\div$ in \cite[Definition~3.5]{AT}, 
we deduce from
\[
\div \bm{b}_R 
= 
-\langle \nabla \psi_R , \nabla u \rangle - \psi_R \Delta u 
= 
-\langle \nabla \psi_R , \nabla u \rangle + \psi_R u 
\]
that $\div \bm{b}_R \in L^\infty(X)$. 
Finally, it follows from
\begin{align*}
&\int_X D^{\mathrm{sym}}\bm{b}_R(\phi_1 , \phi_2 ) \,d\fm \\
&= \frac{1}{2} \int_X \big\{ \psi_R \langle \nabla\phi_1 , \nabla u \rangle \Delta\phi_2
 +\psi_R \langle \nabla \phi_2 , \nabla u \rangle \Delta \phi_1
 -\div (\psi_R \nabla u) \langle \nabla \phi_1 , \nabla \phi_2 \rangle \big\} \,d\fm \\
&= -\frac{1}{2} \int_X \big\{ 2\psi_R \cdot \Hess u( \phi_1 , \phi_2 )
 +\langle \nabla \phi_1 ,\nabla u \rangle \langle \nabla \psi_R , \nabla \phi_2 \rangle
 +\langle \nabla\phi_2 ,\nabla u \rangle \langle \nabla\psi_R,\nabla\phi_1 \rangle \big\} \,d\fm
\end{align*}
(see \cite[(5-3)]{AT} for the definition of $D^{\mathrm{sym}}\bm{b}_R$), 
$\Hess u=0$, $|\nabla u| \le 1$ and $|\nabla \psi_R| \le R^{-1}$ that
\[ \bigg| \int_X D^{\mathrm{sym}}\bm{b}_R( \phi_1 ,\phi_2 ) \,d\fm \bigg|
 \le \frac{1}{R} \int_X | \nabla \phi_1 | | \nabla \phi_2 | \,d\fm
 \le \frac{1}{R} \sqrt{\cE(\phi_1) \cE(\phi_2)}. \]
{Here we have used the fact 
that we can replace $u$ with $\tilde{u} \in \bbD_{\infty} (X)$ 
which is bounded and agrees with $u$ on $B_{3R} (x_0)$ 
by virtue of the presence of $\psi_R$. 
Indeed, since $u$ is Lipschitz, 
$\tilde{u}$ can be constructed by taking a composite of 
an appropriate cut-off function and $u$. 
For $\tilde{u}$ we can use the relation 
between $H[\tilde{u}]$ and $\Hess \tilde{u}$, 
and the chain rule for $\Hess$ implies 
$\Hess u = \Hess \tilde{u}$ on $B_{2R} (x_0)$.}

Next we construct a solution to \eqref{eq:conteq} 
with the aid of Proposition~\ref{pr:meas} for localized initial data. 
Since $\bm{b} := - \nabla u \in L^\infty(X)$, $\exp ( - tu - t^2/2 ) \in L^2(X)$ 
by Lemma~\ref{lm:CD}(ii),
we can apply the superposition principle \cite[Theorem~7.6]{AT}
with $p=2$ and $r=\infty$
(to be precise, \cite[Theorem~7.6]{AT2} with the modified assumptions)
to the solution of the continuity equation for $-\nabla u$ in Proposition~\ref{pr:meas},
to obtain $\bm{\eta} \in \cP ( C ( [ 0, T ]; X ) )$ satisfying
\begin{enumerate}
\item[(a)]
$\bm{\eta}$ is concentrated on solutions $\eta$ to the ODE $\dot{\eta} = \bm{b} (\eta)$
(see \cite[Definition~7.3]{AT}), 

\item[(b)]
$\e^{-tu - t^2/2} \fm = ( e_t )_* \bm{\eta}$ for any $t \in [ 0 , T ]$,
where $e_t(\eta):=\eta(t)$ is the evaluation map. 
\end{enumerate}
Let $r \in ( 0 , R - T )$ and $\bar{v} = \fm ( B_{r} (x_0) )^{-1} \cdot \chi_{ B_r (x_0) }$,
where $\chi_A$ denotes the characteristic function of $A$. 
Then $v_t$ defined by 
$v_t \fm = (e_t)_* (( \bar{v}\circ e_0 ) \bm{\eta} )$
solves the continuity equation for $\bm{b}$ 
with the initial condition $v_0 = \bar{v}$. 
By applying the ODE in (a) for the class of test functions $f_n (x) := d ( x_n , x )$, 
where $\{ x_n \}_{n \in \N} \subset X$ is dense with $| \nabla u | = 1$ 
$\fm$-almost everywhere in mind, 
we can show that $\bm{\eta}$-almost every $\eta$ is 1-Lipschitz 
(see the proof of \eqref{eq:F-isom4} below). 
This fact immediately implies that $v_t$ 
solves the continuity equation for $\bm{b}_R$ also. 
In addition, $v_t = 0$ on $X \setminus B_{r+t} (x_0)$ and thus 
there exists an increasing function $C_r : [ 0, T ] \lra \R$ 
such that, for any measurable set $A \subset X$, 
\begin{align*}
\int_A v_t \, d \fm 
&= 
\int_{A \cap B_{r+t} (x_0)} v_t \, d \fm  
\le
\frac{1}{\fm ( B_r (x_0) )}
\int_{A \cap B_{r+t} (x_0)} d[(e_t)_* \bm{\eta}]
\\
&=
\frac{1}{\fm ( B_r (x_0) )}
\int_{ A\cap B_{r+t} (x_0)} 
\e^{- t u - t^2 /2 }  \, d\fm
\le 
\frac{C_r (t)}{\fm ( B_r (x_0) )} \fm (A) 
\end{align*}
for $t \in [ 0 , T ]$. 
Thus, combining this with the uniqueness of the continuity equation for $\bm{b}_R$ 
starting from $\bar{v}$ as claimed, 
we can apply \cite[Theorem~8.4]{AT} to obtain $\eta_x \in C ([0,T]; X)$
solving the ODE $\dot{\eta}_x = \bm{b}_R(\eta_x)$ with $\eta_x (0) = x$ 
for $(\bar{v} \fm)$-almost every $x$ and satisfying
\[
\bm{\eta} 
= 
\frac{1}{\fm ( B_r (x_0) ) } 
\int_{B_r (x_0)} \delta_{\eta_x} \, \fm (dx). 
\]

We are now in position to follow almost the same argument as in 
\cite[Theorem~8.3]{AT} to conclude our assertion. 
Let us define $F^{(r)} : B_r (x_0) \times [ 0 , T ] \lra X$ 
by $F^{(r)}_t (x) := \eta_x (t)$ for $(\bar{v} \fm)$-almost every $x \in X$. 
We can show the consistency in $r$ of $F^{(r)}$ as in 
\cite[Theorem~8.3]{AT} by using \cite[Theorem~8.4]{AT}, 
by taking larger $R > 0$ if necessary. 
Thus we can let $R \to \infty$ 
to obtain the solution $F:X \times [0,T] \lra X$
satisfying (i), (ii) and (iii). A similar argument allows us to take $T \to \infty$. 
One can further extend this to $F : X \times \R \lra X$ 
by the same construction for $\nabla u$ in parameter $( - \infty , 0 ]$,
and concatenating them. 
The uniqueness of the flow follows similarly and
it implies $( F_t )_* \fm = \e^{- t u - t^2 /2 } \fm$. 
The semigroup property (iv) also follows from the uniqueness.

We finally prove (v). 
On the one hand, 
we already know (ii) and it yields the metric speed of 
$\eta (t) : = F_t (x)$ satisfies $|\dot{\eta}| \le 1$. 
On the other hand, choosing 
$f=\psi_R u \in \mathscr{A}$ 
for (arbitrarily) large $R > 0$ implies 
\[ \frac{d}{dt} \big[ u\big( \eta(t) \big) \big] =-|\nabla u|^2 \big( \eta(t) \big) =-1. \]
Combining this with the $1$-Lipschitz continuity of $u$ yields $|\dot{\eta}| \ge 1$.
Therefore we obtain $|\dot{\eta}|=1$.
$\qedd$
\end{proof}

\begin{remark}[Gaussian behavior of $(F_t)_* \fm$]\label{rm:Gauss}
The relation $(F_t)_* \fm=\e^{-tu-t^2/2}\,\fm$ shows that
$\fm$ is enjoying the `Gaussian' behavior in the $t$-direction.
In fact, when $u(x)=s-t$ (hence $u(F_t^{-1}(x))=s$), we have 
\begin{equation}\label{eq:m_t}
%\e^{-tu(x)-t^2/2} =\e^{-st+t^2/2} =\frac{\e^{-s^2/2}}{\e^{-(s-t)^2/2}}.
%\end{equation}
\e^{-tu-t^2/2} =\frac{\e^{-(u+t)^2/2}}{\e^{-u^2/2}}
\end{equation}
provides the ratio of $e^{-u^2/2}$ and its translation $\e^{-(u+t)^2/2}$.
\end{remark}

\subsection{Behavior of the distance under the flow}\label{ssc:dist}%%%%%
%%%%%%%%%%%%%

The goal of this subsection is 
to show that the regular Lagrangian flow $F$ constructed in Theorem~\ref{th:RLF} 
admits a representative which preserves the distance. 
More precisely, we prove following.

\begin{theorem}[$F_t$ preserves $d$] \label{th:dist}
There exists a map $\widetilde{F} : \R \times X \lra X$ such that
\begin{enumerate}[{\rm (i)}]
\item
\begin{enumerate}[{\rm (a)}]
\item
$\fm (\{ x \in X \,|\, F_t(x) \neq \widetilde{F}_t(x) \; \text{for some $t \in \R$} \}) = 0;$

\item
$\widetilde{F}_t$ is an isometry on $X$ for each $t \in \R;$
\end{enumerate}

\item
$( \widetilde{F}_t (x) )_{t \in \R}$ is a minimal geodesic in $X$ 
for every $x \in X$. 
\end{enumerate}
\end{theorem}

The proof is divided into two propositions below 
(Propositions~\ref{pr:F-isom}, \ref{pr:F-geod}). 
To this end, 
we first lift the flow $F$ on $X$ to the one on $\cP^2 (X)$. 
We remark that, for any $\mu \in \cP^2(X)$,
the curve $t \longmapsto (F_t)_* \mu$ is $1$-Lipschitz in $W_2$
thanks to Theorem~\ref{th:RLF}(ii).

\begin{lemma} \label{lem:C^1}
Let $\mu = \rho \fm \in \cP_{\ac}(X)$ 
where $\rho$ is bounded and of bounded support,  
and $\mu_t : = ( F_t )_* \mu$. 
\begin{enumerate}[{\rm (i)}]
\item \label{lem:C^1-1}
We have $\mu_t = (\rho \circ F_{-t}) \e^{-t u - t^2/2} \fm$ 
for all $t \in \R$.  
In particular, $\mu_t \ll \fm$ and 
the density of $\mu_t$ is bounded and of bounded support. 

\item \label{lem:C^1-2}
Suppose that, for $\fm$-almost every $x \in X$, 
$\rho ( F_t (x) )$ is continuous in $t \in \R$. 
Then, for any $ f\in W^{1,2}_{\loc}(X)$,
the function $t \longmapsto \int_X f \,d\mu_t$ belongs to 
$C^1 (\R)$ and we have
\[
\frac{d}{dt} \int_X f \, d \mu_t 
= 
-\int_X \langle \nabla f , \nabla u \rangle \, d \mu_t . 
\]
\end{enumerate}
\end{lemma}

\begin{proof}
(\ref{lem:C^1-1}) By Theorem~\ref{th:RLF}, 
for any bounded measurable $f : X \lra \R$, we find
\[
\int_X f \, d \mu_t 
 = 
\int_X (f \circ F_t) \rho \, d \fm 
 = 
\int_X \big( f (\rho \circ F_{-t} ) \big) \circ F_t \, d \fm 
= 
\int_X f (\rho \circ F_{-t}) \e^{-tu - t^2/2} \, d \fm . 
\]
It immediately implies the former assertion. 
The latter one easily follows from the assumption on $\mu$ 
and Theorem~\ref{th:RLF}(ii). 

(\ref{lem:C^1-2})
Since (i) says that $\mu_t$ has a bounded support for each $t \in \R$, 
we can assume $f \in W^{1,2} (X)$ without loss of generality.
By virtue of Theorem~\ref{th:RLF}(iii), it suffices to show that
\[
t \longmapsto \int_X \langle \nabla f,\nabla u \rangle \, d\mu_t 
\]
is continuous. 
Since $| \nabla u | = 1$ $\fm$-almost everywhere, 
we can easily deduce it from \cite[Lemma~5.11]{Gi-split} 
with the aid of (i), Theorem~\ref{th:RLF}(ii) and our assumption on $\rho$. 
$\qedd$
\end{proof}

Recall the function $\cU(\mu)=\int_X u \,d\mu$ in the previous section \eqref{eq:U},
which is affine on $\cP^2_{\ac}(X)$ by Theorem~\ref{th:affine}.
The next lemma will play a key role in this section.

\begin{lemma}[Evolution variational equality for $\cU$] \label{lem:w-EVI}
Let $\mu \in \cP (X)$ be of bounded support 
with bounded continuous density, 
and $\mu_t : = ( F_t )_* \mu$. 
Then $\mu_t$ solves the \emph{$0$-evolution variational equality} for $\cU$
in the sense that, for any $\nu \in \cP^2(X)$,
\begin{equation}
\label{eq:w-evi}
\frac{d}{dt} \frac{W_2^2(\mu_t,\nu)}{2}
 = \cU ( \nu ) - \cU ( \mu_t )
% \int_X u \,d\nu -\int_X u \,d\mu_t 
\end{equation}
holds at almost every $t \in \R$.
\end{lemma}

The \emph{$0$-evolution variational inequality} (abbreviated as the \emph{$0$-EVI})
means that
\[ \frac{d}{dt} \frac{W_2^2(\mu_t,\nu)}{2}
 \le \cU ( \nu ) - \cU ( \mu_t ). \]
We will obtain equality as in \eqref{eq:w-evi} due to the symmetry between $u$ and $-u$,
thus we called it the $0$-evolution variational equality.

\begin{proof}
From Theorem~\ref{th:RLF}(ii), 
we deduce that $W_2(\mu_t,\mu_s) \le |s-t|$
so that $t \longmapsto \mu_t$ is $W_2$-absolutely continuous.
Thanks to \cite[Proposition~2.21(i)]{AGS-rcd}, 
it suffices to show \eqref{eq:w-evi} for $\nu \in \cP (X)$
of bounded support with bounded density. 

Let $( \varphi_t, \psi_t )$ be a Kantorovich potential for $( \mu_t , \nu )$ (recall \eqref{eq:Kant}),
namely
\begin{align*}
\frac{1}{2} W_2^2(\mu_t,\nu)
&= \int_X \varphi_t \,d\mu_t -\int_X \psi_t \,d\nu, \\
\varphi_t(x) -\psi_t(y) &\le \frac{d^2(x,y)}{2} \quad \text{for all}\ x,y \in X.
\end{align*}
We first claim that, 
for a point $t$ of differentiability of $t \longmapsto W_2^2(\mu_t,\nu)$, 
we have
\begin{equation}\label{eq:dW/dt}
\frac{d}{dt} 
\frac{W_2^2(\mu_t,\nu)}{2}
 =
- 
\int_X 
  \langle \nabla u , \nabla \varphi_t \rangle 
\,d\mu_t . 
\end{equation}
Since both $\mu_t$ and $\nu$ have bounded support, 
we can assume that $\varphi_t$ is Lipschitz (see \cite[Lemma~2]{Mc} for instance)
and hence $\varphi_t \in W^{1,2}_{\loc} (X)$. 
The Kantorovich duality \eqref{eq:Kant} immediately implies 
\begin{align*}
\frac{W_2^2(\mu_{t+s} ,\nu)}{2} 
& \ge 
\int_X \varphi_t \, d\mu_{t+s} - \int_X \psi_t \, d \nu , 
& 
\frac{W_2^2 ( \mu_t , \nu )}{2}
& = 
\int_X \varphi_t \, d \mu_t - \int_X \psi_t \, d \nu . 
\end{align*}
By combining them, we have 
\begin{align*}
\frac{d}{dt} 
\frac{W_2^2(\mu_t,\nu)}{2}
&= 
\lim_{s\downarrow 0}
\frac{W_2^2(\mu_{t+s} ,\nu) - W_2^2 ( \mu_t , \nu )}{2s}
\\
&\ge
\lim_{s \downarrow 0} \frac{1}{s} 
\left( 
     \int_X \varphi_t \, d \mu_{t+s} 
     - 
     \int_X \varphi_t \, d \mu_t 
\right)
= 
- 
\int_X 
  \langle \nabla u , \nabla \varphi_t \rangle 
\,d\mu_t,
\end{align*}
where the last inequality follows 
from Lemma~\ref{lem:C^1}(\ref{lem:C^1-2}). 
We similarly observe 
\[
\frac{d}{dt} 
\frac{W_2^2(\mu_t,\nu)}{2}
= 
\lim_{s\downarrow 0}
\frac{W_2^2(\mu_{t} ,\nu) - W_2^2 ( \mu_{t-s} , \nu )}{2s}
\le
- 
\int_X 
  \langle \nabla u , \nabla \varphi_t \rangle 
\,d\mu_t . 
\]
Thus \eqref{eq:dW/dt} holds. 

Next we prove \eqref{eq:w-evi}. 
Let $( \nu_s )_{s \in [0,1]}$ be the unique $W_2$-geodesic from $\mu_t$ to $\nu$.
Then, since the density of $\mu$ is continuous, 
we can apply \cite[Proposition~5.15]{Gi-split} 
to deduce that $s \longmapsto \int_X u \,d\nu_s$ is 
differentiable at $s=0$ and
\begin{equation} \label{eq:du-geod}
\left. \frac{d}{ds} \right|_{s=0} \int_X u \,d\nu_s
 =
-\int_X \langle \nabla u,\nabla\varphi_t \rangle \,d\mu_t.
\end{equation}
We finally recall from Theorem~\ref{th:affine} that
$\cU ( \nu_s ) =(1-s)\cU (\nu_0) +s\cU(\nu_1)$,
therefore
\[
\left. \frac{d}{ds} \right|_{s=0} \int_X u \,d\nu_s
= \cU(\nu) -\cU(\mu_t).
\]
This together with \eqref{eq:dW/dt} and \eqref{eq:du-geod} 
yields \eqref{eq:w-evi}. 
$\qedd$
\end{proof}

\begin{remark}\label{rm:w-EVI}
In Lemma~\ref{lem:w-EVI}, 
the equality \eqref{eq:w-evi} in fact holds for all $t \in \R$ 
since $W_2^2( \mu_t , \nu )$ is locally Lipschitz
and $\cU(\mu_t)$ is continuous in $t$. 
\end{remark}

From Lemma~\ref{lem:w-EVI},
we deduce that the flow given by $F$ preserves $W_2$. 

\begin{lemma}[$F_t$ preserves $W_2$] \label{lem:isom}
Let $\mu, \nu \in \cP^2_{\ac} (X)$ with continuous bounded densities. 
Set $\mu_t : = ( F_t )_* \mu$ and $\nu_t := ( F_t )_* \nu$. 
Then we have
\begin{equation} \label{eq:isom}
W_2 ( \mu_t, \nu_t ) = W_2 ( \mu , \nu )
 \qquad \text{for all}\ t \in \R.
\end{equation}
\end{lemma}

\begin{proof}
Let us first additionally suppose that 
$\mu$ and $\nu$ have bounded supports, 
then Lemma~\ref{lem:w-EVI} is available. 
Since $W_2 ( \mu_t , \nu_t )$ is $2$-Lipschitz in $t$,
it suffices to show that $t \longmapsto W_2^2 ( \mu_t , \nu_t )$ has 
a vanishing derivative for almost every $t$. 
Let $\omega$ be the midpoint of the $W_2$-geodesic 
from $\mu_t$ and $\nu_t$. 
Then, by \eqref{eq:w-evi}, 
\begin{align*}
&\varlimsup_{\ve \downarrow 0}
\frac{W_2^2(\mu_{t+\ve},\nu_{t+\ve}) -W_2^2(\mu_t,\nu_t)}{2\ve}
 \le 
\varlimsup_{\ve \downarrow 0} 
\frac{2W_2^2(\mu_{t+\ve},\omega)+2W_2^2(\omega,\nu_{t+\ve})-W_2^2(\mu_t,\nu_t)}{2\ve} \\
&= 
\varlimsup_{\ve \downarrow 0}
\left( 
\frac{W_2^2(\mu_{t+\ve},\omega) -W_2^2(\mu_t,\omega)}{\ve}
 +
\frac{W_2^2(\omega,\nu_{t+\ve}) -W_2^2(\omega,\nu_t)}{\ve} 
\right)
\\
& = 
4 
\bigg( 
  \int_X u \,d\omega 
  -\frac{1}{2} \int_X u \,d\mu_t 
  -\frac{1}{2} \int_X u \,d\nu_t 
\bigg)
=0.
\end{align*}
Here the last equality follows from the affine property of $\cU$ 
(Theorem~\ref{th:affine}).
By the same way, we have
\[ 
\varliminf_{\ve \downarrow 0}
 \frac{W_2^2(\mu_t,\nu_t) -W_2^2(\mu_{t-\ve},\nu_{t-\ve})}{2\ve} 
\ge 0. 
\]
Therefore \eqref{eq:isom} holds for every $t \in \R$. 

We next remove the assumption on bounded support 
by a standard cut-off argument. 
Let $x_0 \in X$ and $\psi_n : X \lra \R$ be continuous 
satisfying $0 \le \psi_n \le 1$, $\psi_n |_{B_n (x_0)} = 1$ 
and $\psi_n |_{X \setminus B_{n+1} (x_0)} = 0$. 
Let us define $\mu^{(n)},\nu^{(n)} \in \cP^2_{\ac}(X)$ 
for $n \in \N$ as follows: 
\[
\mu^{(n)} : = \left( \int_X \psi_n \, d \mu \right)^{-1} \psi_n \cdot \mu,
\qquad
\nu^{(n)} : =\left( \int_X \psi_n \, d \mu \right)^{-1} \psi_n \cdot \nu. 
\]
We can easily see $W_2 ( \mu^{(n)} , \mu ) \to 0$ as $n \to \infty$ 
(see \cite[Proposition~7.1.5]{AGS-book} for instance). 
Thus \eqref{eq:isom} implies that 
$\{ ( F_t )_* \mu^{(n)} \}_{n \in \N}$ 
forms a $W_2$-Cauchy sequence. 
For each bounded $f \in C (X)$, 
the dominated convergence theorem yields 
\[
\lim_{n \to \infty} 
\int_X f \, d [( F_t )_* \mu^{(n)}] 
= 
\lim_{n \to \infty} 
\int_X f \circ F_t \, d \mu^{(n)} 
= 
\int_X f \circ F_t \, d \mu
= 
\int_X f \, d \mu_t . 
\]
Thus $W_2 ( ( F_t )_* \mu^{(n)} , \mu_t ) \to 0$ as $n \to \infty$ 
(again by \cite[Proposition~7.1.5]{AGS-book}),
and similarly $W_2 ( ( F_t )_* \nu^{(n)} , \nu_t ) \to 0$.
Thus the conclusion holds by applying \eqref{eq:isom}
to $( \mu^{(n)} , \nu^{(n)} )$ and letting $n \to \infty$. 
$\qedd$
\end{proof}

We are now ready to prove Theorem~\ref{th:dist}. 
We first deduce from Lemma~\ref{lem:isom} that $F_t$ is an isometry. 
Note that we may not have Lebesgue points 
since $\fm$ is not necessarily doubling. 
Thus we will follow an alternative strategy. 
Roughly speaking, the idea is to consider \eqref{eq:isom}
in the \emph{Kantorovich--Rubinstein duality}:
\[ 
W_1(\mu,\nu)
=
\sup \bigg\{ 
\int_X \varphi \,d\mu -\int_X \varphi \,d\nu
 \,\bigg|\, 
\varphi : X \to \R \text{, $1$-Lipschitz} 
\bigg\} 
\]
with some approximation.

\begin{proposition}[Proof of Theorem~\ref{th:dist}(i)] \label{pr:F-isom}
There exists $\widetilde{F}: \R \times X \lra X$ such that 
Theorem~$\ref{th:dist}${\rm (i)} holds. 
\end{proposition}

\begin{proof}
Fix $t \in \R$ and a bounded $1$-Lipschitz function $f : X \lra \R$. 
We first show that $f \circ F_t$ has 
a $1$-Lipschitz representative 
in its $\fm$-almost everywhere equivalence class. 
In order to see this, we consider $g_\ve:=\sH_\ve(f\circ F_t)$ for $\ve>0$. 
Recall that $g_\ve$ is Lipschitz by Proposition~\ref{pr:Lip}.
Pick $x,y \in X$, $r>0$ and $\mu_r,\nu_r \in \cP^2_{\ac}(X)$ 
with bounded continuous density supported 
on $B_r (x)$ and $B_r (y)$, respectively. 
Then, since $f$ is $1$-Lipschitz,  
\begin{align} \nonumber
\bigg| \int_X g_\ve \,d\mu_r -\int_X g_\ve \,d\nu_r \bigg|
&=
\bigg| \int_X f \circ F_t \,d[\sH_\ve(\mu_r)] -\int_X f \circ F_t \,d[\sH_\ve(\nu_r)] \bigg| 
\\ \nonumber
&=
\bigg| \int_X f \,d[(F_t)_* \sH_\ve(\mu_r)] -\int_X f \,d[(F_t)_* \sH_\ve(\nu_r)] \bigg| 
\\ \nonumber
& \le
W_1 \big( (F_t)_* \sH_\ve(\mu_r),(F_t)_* \sH_\ve(\nu_r) \big)
\\ \label{eq:F-isom1}
& \le
W_2 \big( (F_t)_* \sH_\ve(\mu_r),(F_t)_* \sH_\ve(\nu_r) \big) .
\end{align}
We used the Kantorovich--Rubinstein duality and the H\"older inequality
to see the inequalities above.
Note that $\sH_\ve(\mu_r)$ and $\sH_\ve(\nu_r)$ 
also have bounded continuous density by Proposition~\ref{pr:Lip}. 
Thus it follows from Lemma~\ref{lem:isom} and the $W_2$-contraction property \eqref{eq:Wcont}
of the heat flow that
\[
W_2 \big( (F_t)_* \sH_\ve(\mu_r),(F_t)_* \sH_\ve(\nu_r) \big) 
 = 
W_2 \big( \sH_\ve(\mu_r), \sH_\ve(\nu_r) \big) 
 \le
\e^{-\ve} W_2(\mu_r,\nu_r).
\]
Combining this with \eqref{eq:F-isom1}
and letting $r\downarrow0$, 
we obtain 
\begin{equation} \label{eq:F-isom3}
| g_\ve (x) - g_\ve (y) | \leq \e^{-\ve} d(x,y). 
\end{equation}
Since $g_\ve$ converges to $f \circ F_t$ 
in $L^2 (X)$ as $\ve \downarrow 0$, 
by taking an almost everywhere converging subsequence, 
we obtain $| f \circ F_t (x) - f \circ F_t (y) | \le d ( x, y )$ 
for $\fm \otimes \fm$-almost every $(x,y)$ 
from \eqref{eq:F-isom3}. 
It implies our claim. 

We next show that, for each $t \in \R$, there 
exists a Borel $\fm$-negligible set $A \subset X$ 
such that the following holds: 
\begin{equation} \label{eq:F-isom4}
d\big( F_t (x) , F_t (y) \big) = d ( x, y ) 
\qquad \text{for any $x,y \in X \setminus A$.}
\end{equation}
Let $\{x_i\}_{i \in \N}$ be a countable dense subset in $X$ 
and let $f_i(x):=d(x,x_i)$. 
A truncation argument shows that we can remove 
the boundedness of $f$ from the assumption in the last claim. 
Thus we can apply it to $f_i$ to conclude that 
there exist a Borel $\fm$-negligible set $A \subset X$ such that 
$f_i \circ F_t$ is $1$-Lipschitz on $X \setminus A$ for all $i \in \N$. 
Thus we have
\[
d\big( F_t(x),F_t(y) \big)
 =\sup_i \big\{ f_i \big( F_t(x) \big) -f_i \big( F_t(y) \big) \big\}
 \le d(x,y)
\]
for all $x,y \in X \setminus A$,
which proves that the restriction of $F_t$ to $X \setminus A$ is $1$-Lipschitz.
Then, Theorem~\ref{th:RLF}(i), (iv) imply \eqref{eq:F-isom4} 
by exchanging $t$ with $-t$ in the above argument. 

Finally, we construct a modification $\widetilde{F}$ of $F$. 
From the last argument, there exists a Borel $\fm$-negligible 
subset $A \subset X$ such that \eqref{eq:F-isom4} holds 
for any $t \in \mathbb{Q}$. 
By Theorem~\ref{th:RLF}(ii), the same holds for any $t \in \R$. 
Then, for each $t \in \R$, we have the unique 
extension $\widetilde{F}_t$ of $F_t$ as an isometry. 
This completes the proof. 
$\qedd$
\end{proof}

We can also improve Theorem~\ref{th:affine} as follows.

\begin{proposition}[$u$ is affine]\label{pr:u-affine}
The function $u$ is affine in the sense that,
along any geodesic $\gamma:[0,1] \lra X$,
we have for all $t \in (0,1)$
\[ u\big( \gamma(t) \big)
 =(1-t)u\big( \gamma(0) \big) +tu\big( \gamma(1) \big). \]
\end{proposition}

\begin{proof}
Let $x,y \in X$ and consider $\mu,\nu \in \cP^2_{\ac}(X)$
of bounded continuous density and bounded support
approximating the Dirac measures $\delta_x$ and $\delta_y$ in the sense of weak convergence, respectively. Consider the map $\tilde{F}:\mathbb{R}\times X\rightarrow X$ and 
define $\tilde{\mu}_t:=(\tilde{F}_t)_{\star}\mu$.
Since by the previous theorem for every $t\in\mathbb{R}$ $F_t=\tilde{F}_t$ $\fm$-a.e., we have $\mu_t=\tilde{\mu}_t$.
Then, by integrating \eqref{eq:w-evi} in $t$ we obtain
\begin{equation}
\frac{W_2^2(\tilde{\mu}_t,\nu)}{2}-\frac{W_2^2({\mu},\nu)}{2}
 = \cU ( \nu )t - \int_0^t\cU (\tilde{\mu}_{\tau} )d\tau
% \int_X u \,d\nu -\int_X u \,d\mu_t 
\end{equation}
Finally, since $\tilde{F}_t$ is continuous for every $t\in \mathbb{R}$, one can pass to the limit 
as $\mu \to \delta_x$ and $\nu \to \delta_y$.
We deduce that $\eta(t):=\tilde{F}_t(x)$ enjoys the 0-evolution variational equality for $u$:
\begin{equation}\label{eq:u-EVI}
\frac{d}{dt} \frac{d^2(\eta(t),y)}{2}
 = u(y) -u\big( \eta(t) \big).
\end{equation}
This implies that both $u$ and $-u$ are convex,
and hence affine.
$\qedd$
\end{proof}

The next proposition completes the proof of Theorem~\ref{th:dist}. 
The key fact in the proof is that $( \widetilde{F}_t (x) )_{t \in \R}$ 
provides the EVI-gradient flow of $u$ as we saw in Proposition~\ref{pr:u-affine}.

\begin{proposition}[Proof of Theorem~\ref{th:dist}(ii)] \label{pr:F-geod}
For each $x \in X$, the curve
$( \widetilde{F}_t (x) )_{t \in \R}$ is a minimal geodesic in $X$. 
\end{proposition}

\begin{proof}
Take $x \in X$ to be a point 
such that $F_t (x) = \widetilde{F}_t (x)$ for all $t \in \R$,
the property in Theorem~\ref{th:RLF}(v) holds,
and that $|\nabla^L u|(F_t(x)) \ge |\nabla u|(F_t(x))=1$ for almost every $t \in \R$.
Notice that the validity of the last property is ensured by 
Theorem~\ref{th:RLF}(i) and the Fubini theorem for localized measures. 
Recall from the proof of Proposition~\ref{pr:u-affine} that
$\eta(t):=\tilde{F}_t(x)$ enjoys the $0$-evolution variational equality \eqref{eq:u-EVI} for $u$.
On the one hand,
since EVI-gradient flows are gradient flows also in the sense of 
the \emph{energy dissipation identity} (the proof of this fact, due to Savar\'e, can be found in \cite{AG}), 
we have for every $s<t$
\[
u \big( \eta(t) \big) - u \big( \eta(s) \big) 
= 
- \frac{1}{2} \int_s^t \big\{ |\nabla^L u|^2 \big( \eta(r) \big) +|\dot{\eta}|^2(r) \big\} \,dr
\le s-t,
\]
where we used Theorem~\ref{th:RLF}(v) to see $|\dot{\eta}|=1$.
On the other hand,
\[
\big| u\big( F_t (x)) - u(F_s (x) \big) \big| \le d\big( F_t (x), F_s (x) \big) \le |t-s|
\]
holds since $u$ and $(F_t(x))_{t \in \R}$ are 1-Lipschitz,
and thus $d ( F_t (x), F_s (x) ) = |t - s|$ 
for every $t, s \in \R$. 
This forces the curve 
$( F_t(x) )_{t \in \R} = ( \widetilde{F}_t (x) )_{t \in \R}$ 
to be a minimal geodesic (straight line) in $X$. 
Since $\widetilde{F}_t$ is a continuous map on $X$ 
for each $t$ by Proposition~\ref{pr:F-isom}, 
$( \widetilde{F}_t (x) )_{t \in \R}$ must be a geodesic for every $x \in X$. 
$\qedd$
\end{proof}

% \begin{remark}\label{rm:bdry}
% It is somewhat implicit in our discussion that the sharp spectral gap prevents
% spaces ``with boundary'' such as $Y \times [0,\infty)$ showing up
% (while $Y$ of $Y \times \R$ can have a boundary).
% Indeed, on $Y \times [0,\infty)$, the function $u(y,t)=t$ is not an eigenfunction
% since its measure-valued Laplacian has singularity on $Y \times \{0\}$.
% \end{remark}

\section{Third step: Isometric splitting}\label{sc:isom}%%%%%
%%%%%%%%%%%%%%%%%%%%%

The properties of the gradient flow $({\tilde{F}}_t)_{t \in \R}$
of the eigenfunction $-u$ obtained in the previous section
allow us to follow the strategy of the splitting theorem in \cite{Gi-split,Gi-ov} to a large extent.

Set $Y:=u^{-1}(0)$.
The affine property of $u$ (Proposition~\ref{pr:u-affine}) implies that
$Y$ is totally geodesic in the sense that
any geodesic connecting two points in $Y$ is contained in $Y$.
Thus the distance $d_Y:=d|_{Y \times Y}$ on $Y$
defined as the restriction is geodesic.
We would like to compare $X$ and $Y \times \R$.
To this end, we define the maps
\begin{align*}
\pi &:X \ni x \,\longmapsto\, {\tilde{F}}_{u(x)} (x) \in Y, \\
\Phi &:X \ni x \,\longmapsto\, \big( \pi(x),-u(x) \big) \in Y \times \R, \\
\Psi &:Y \times \R \ni (y,t) \,\longmapsto\, {\tilde{F}}_t (y) \in X.
\end{align*}
Notice that $\pi$ is well-defined since $u({\tilde{F}_{t}}(x))=u(x)-{t}$ 
{for $x \in X$ and $t \in \R$} by Theorem~\ref{th:RLF}(iii) 
{and Theorem~\ref{th:dist}}.
We have by construction $\Psi=\Phi^{-1}$.
We first prove an important property of the map $\pi$
along the strategy in \cite[Corollary~5.19]{Gi-split} (see also \cite[Corollary~4.6]{Gi-ov}).

\begin{lemma}\label{lm:pi-1lip}
The map $\pi$ is $1$-Lipschitz
\end{lemma}

\begin{proof}
Since ${\tilde{F}}_t$ is isometric for each $t \in \R$,
we find $d(x,x') =d({\tilde{F}}_{u(x')}(x),\pi(x'))$.
Thus it is sufficient to show
$d(\pi(x),y) \le d(x,y)$ for $x \in X$ and $y \in Y$.
Fix $y \in Y$ and $\mu \in \cP^2_{\ac}(X)$ with bounded density,
and consider $\mu_t:=({\tilde{F}}_t)_* \mu$.
Take $t_0 \in \R$ attaining the minimum of the function $t \longmapsto W_2^2(\mu_t,\delta_y)$.
Let $(\nu_s)_{s \in [0,1]}$ be the minimal geodesic from $\mu_{t_0}$ to $\delta_y$.
Then, for every $s \in (0,1)$ and $t \in \R$, we find
\begin{align*}
W_2(\nu_s,\delta_y)
&=(1-s) W_2(\mu_{t_0},\delta_y)
 \le (1-s) W_2\big( ({\tilde{F}}_t)_* \mu_{t_0},\delta_y \big) \\
&\le (1-s) \big\{ W_2 \big( ({\tilde{F}}_t)_* \mu_{t_0},({\tilde{F}}_t)_* \nu_s \big) +W_2\big( ({\tilde{F}}_t)_* \nu_s,\delta_y \big) \big\} \\
&= (1-s) \big\{ W_2(\mu_{t_0},\nu_s) +W_2\big( ({\tilde{F}}_t)_* \nu_s,\delta_y \big) \big\} \\
&= s W_2(\nu_s,\delta_y) +(1-s) W_2\big( ({\tilde{F}}_t)_* \nu_s,\delta_y \big).
\end{align*}
Thus $W_2(({\tilde{F}}_t)_* \nu_s,\delta_y)$ attains the minimum at $t=0$.

Put $\varphi(x):=d^2(x,y)/2$
which is a Kantorovich potential for $(\nu_s,\delta_y)$ for all $s$.
Then it follows from \eqref{eq:dW/dt} that
\[ 0=\frac{d}{dt} \frac{W_2^2(({\tilde{F}}_t)_* \nu_s,\delta_y)}{2} \bigg|_{t=0}
 =-\int_X \langle \nabla u,\nabla \varphi \rangle \,d\nu_s \]
for all $s \in [0,1]$.
This yields (by \cite[Proposition~5.15]{Gi-split})
\[ \lim_{h \downarrow 0} \frac{1}{h}
 \bigg\{ \int_X u \,d\nu_{s+h} -\int_X u \,d\nu_s \bigg\}
 = -\frac{1}{1-s} \int_X \langle \nabla u,\nabla \varphi \rangle \,d\nu_s=0, \]
therefore
\[ \int_X u \,d\mu_{t_0} =\int_X u \,d\nu_0
 =\lim_{s \uparrow 1} \int_X u \,d\nu_s =u(y)=0. \]
This means that, by taking $\mu$ converging to $\delta_x$,
the minimum of $t \longmapsto d({\tilde{F}}_t(x),y)$ is attained at $t=u(x)$.
Hence we have $d(\pi(x),y) \le d(x,y)$.
$\qedd$
\end{proof}

On $Y \times \R$ let us consider the $L^2$-product distance:
\[ \hat{d} \big( (y_1,s),(y_2,t) \big) :=\sqrt{d_Y^2(y_1,y_2) +|s-t|^2}
 \qquad \text{for}\ (y_1,s),(y_2,t) \in Y \times \R. \]
Then it is easily seen that $\Phi$ and $\Psi$ are Lipschitz,
thus they give a bi-Lipschitz homeomorphism
(see \cite[Proposition~5.26]{Gi-split}, \cite[Proposition~4.9]{Gi-ov}).

\begin{lemma}[$\Phi$ and $\Psi$ are Lipschitz]\label{lm:biLip}
For any $(y_1,s),(y_2,t) \in Y \times \R$, we have
\[ \frac{1}{2} \hat{d}^2\big( (y_1,s),(y_2,t) \big)
 \le d^2 \big( \Psi(y_1,s),\Psi(y_2,t) \big)
 \le 2\hat{d}^2 \big( (y_1,s),(y_2,t) \big). \]
\end{lemma}

\begin{proof}
The first inequality follows from the fact that
both $\pi$ and $u$ are $1$-Lipschitz:
\begin{align*}
d^2 \big( \Psi(y_1,s),\Psi(y_2,t) \big)
 &\ge \max\big\{ d_Y^2 \big( \pi \circ \Psi(y_1,s),\pi \circ \Psi(y_2,t) \big),
 |u \circ \Psi(y_1,s) -u \circ \Psi(y_2,t)|^2 \big\} \\
 &= \max\big\{ d_Y^2(y_1,y_2),|s-t|^2 \big\}
 \ge \frac{1}{2} \Big( d_Y^2(y_1,y_2) +|s-t|^2 \Big).
\end{align*}
The second inequality is a consequence of the properties of ${\tilde{F}}_t$:
\begin{align*}
d \big( \Psi(y_1,s),\Psi(y_2,t) \big)
&= d \big( {\tilde{F}}_0(y_1),{\tilde{F}}_{t-s}(y_2) \big) \\
&\le d \big( {\tilde{F}}_0(y_1),{\tilde{F}}_0(y_2) \big) +d \big( {\tilde{F}}_0(y_2),{\tilde{F}}_{t-s}(y_2) \big) \\
&= d_Y(y_1,y_2) +|t-s|
 \le \sqrt{2 \Big( d_Y^2(y_1,y_2) +|t-s|^2 \Big)}.
\end{align*}
$\qedd$
\end{proof}

Define the measure $\fm_Y$ on $Y$ by
\[ \fm_Y(A) :=\lim_{\ve \to 0} \frac{\fm(\Psi(A \times [0,\ve]))}{\ve}. \]
By the relation $({\tilde{F}}_t)_* \fm=\e^{-tu-t^2/2} \,\fm$ obtained in Theorem~\ref{th:RLF},
we see that (recall also \eqref{eq:m_t}) the limit exists and
\begin{equation}\label{eq:m-split}
d[\Phi_* \fm] =d\fm_Y \times (\e^{-t^2/2} \,dt).
\end{equation}

What is remaining is the relation between $d$ on $X$ and $\hat{d}$ on $Y \times \R$.
We first observe the following by the same argument as
\cite[Corollary~5.30]{Gi-split}, \cite[Corollary~4.12]{Gi-ov}.

\begin{lemma}\label{lm:Y-RCD}
$(Y,d_Y,\fm_Y)$ satisfies $\RCD(1,\infty)$.
\end{lemma}

\begin{proof}
First, in order to see the infinitesimal Hilbertianity,
let us extend $\tilde{f},\tilde{g} \in W^{1,2}_{\loc}(Y)$ to $X$ as
$f:=\tilde{f} \circ \pi, g:=\tilde{g} \circ \pi$, respectively.
Then $f,g \in W^{1,2}_{\loc}(X)$ and the infinitesimal Hilbertianity for $f,g$
and \eqref{eq:m-split} shows the claim.

Next, to prove $\CD(1,\infty)$, we consider the map
$\Xi: \cP^2(Y) \lra \cP^2(X)$ defined by
\[ 
\Xi(\mu) := 
\Psi_* ( \mu \times \cL^1|_{[0,1]} ), 
\]
where 
$\cL^1$ is the Lebesgue measure.
Then, since ${\tilde{F}}_t$ is isometric and $Y$ is totally geodesic, we deduce that
$\Xi$ is isometric (compare with the proof of Corollary 5.30 in \cite{Gi-split}) and, for any $\mu_0,\mu_1 \in \cP^2_{\ac}(Y)$
and the unique geodesic $(\mu_t)_{t \in [0,1]}$ between $\mu_0$ and $\mu_1$,
$\Xi(\mu_t)$ is being the minimal geodesic between $\Xi(\mu_0)$ and $\Xi(\mu_1)$.
Hence the curvature condition $\CD(1,\infty)$ of $(X,d,\fm)$
applied to $\Xi(\mu_t)$ implies $\CD(1,\infty)$ for $\mu_t$.
$\qedd$
\end{proof}

As a corollary to the lemma above, the product space
\[ \big( Y \times \R,\hat{d},d\fm_Y \times (\e^{-t^2/2} \,dt) \big) \]
again satisfies $\RCD(1,\infty)$.
The following energy identity is the key ingredient to show $d=\hat{d}$.
The proof follows the same line as \cite[Proposition~4.15]{Gi-ov}
and \cite[Proposition~6.5]{Gi-split},
we refer to those for the details of the discussion.

\begin{proposition}[Energy identity]\label{pr:E=E}
For all $f \in L^2(Y \times \R)$, we have
\[ \cE_X(f \circ \Phi) =\cE_{Y \times \R}(f). \]
\end{proposition}

\begin{proof}
By the density reasons (for instance, compare with \cite{GH}),
it is sufficient to show the claim for functions of the form
\[ f=\sum_{i \in I} g_i h_i \]
for a finite set $I$ and $g_i \in \mathscr{G}, h_i \in \mathscr{H}$,
where
\begin{align*}
\mathscr{G} &:=\{ g:Y \times \R \lra \R \,|\,
 g(y,t)=\tilde{g}(y)\ \text{for some}\ \tilde{g} \in W^{1,2} \cap L^{\infty}(Y) \}, \\
\mathscr{H} &:=\{ h:Y \times \R \lra \R \,|\,
 h(y,t)=\tilde{h}(t)\ \text{for some}\ \tilde{h} \in W^{1,2} \cap L^{\infty}(\R) \}.
\end{align*}
Recalling that $Y \times \R$ is an $\RCD(1,\infty)$-space,
we expand $|\nabla f|_{Y \times \R}^2$ as
\[ |\nabla f|_{Y \times \R}^2
 = \sum_{i,j \in I} \Big\{ g_i g_j \langle \nabla h_i,\nabla h_j \rangle_{Y \times \R}
 +2g_i h_j \langle \nabla h_i,\nabla g_j \rangle_{Y \times \R} \\
 +h_i h_j \langle \nabla g_i,\nabla g_j \rangle_{Y \times \R} \Big\}. \]
In order to compare this with the same decomposition of $f \circ \Phi$,
notice that by the very same arguments used in Gigli's proof of the splitting theorem \cite{Gi-split} we have that
\[ |\nabla g|_{Y \times \R} \circ \Phi =|\nabla (g \circ \Phi)|
 \qquad \text{$\fm$-almost everywhere} \]
for all $g \in \mathscr{G}$ and, similarly,
\[ |\nabla h|_{Y \times \R} \circ \Phi =|\nabla (h \circ \Phi)|
 \qquad \text{$\fm$-almost everywhere} \]
for all $h \in \mathscr{H}$.
Thus we have
\begin{align*}
\langle \nabla g_i,\nabla g_j \rangle_{Y \times \R} 
 \circ \Phi
& =
\langle 
  \nabla ( g_i \circ \Phi ) , 
  \nabla ( g_j \circ \Phi ) 
\rangle_X, 
\\
\langle \nabla h_i,\nabla h_j \rangle_{Y \times \R}
 \circ \Phi
& =
\langle 
  \nabla ( h_i \circ \Phi ) , 
  \nabla ( h_j \circ \Phi ) 
\rangle_X 
\end{align*}
$\fm$-almost everywhere by polarization.

Now it suffices to prove that, for any $g \in \mathscr{G}$ and $h \in \mathscr{H}$,
\begin{align}
\langle \nabla g,\nabla h \rangle_{Y \times \R}=0
 \qquad & (\fm_Y \times \cL^1)\text{-almost everywhere},
 \label{eq:E=E1}\\
\langle \nabla (g \circ \Phi),\nabla (h \circ \Phi) \rangle_X=0
 \qquad & \fm\text{-almost everywhere}.
 \label{eq:E=E2}
\end{align}
The former relation \eqref{eq:E=E1} follows from the product structure of $Y \times \R$,
see \cite[Theorem~5.1]{AGS-be}.
In order to see the latter \eqref{eq:E=E2},
let us take $\tilde{h} \in W^{1,2} \cap L^{\infty}(\R)$ with $h(y,t)=\tilde{h}(t)$
and notice by the definition of $\Phi$ that $h \circ \Phi =\tilde{h} \circ (-u)$.
Hence
\[ \langle \nabla(g \circ \Phi),\nabla(h \circ \Phi) \rangle_X
 =-\tilde{h}' \circ (-u) \cdot \langle \nabla(g \circ \Phi),\nabla u \rangle_X. \]
Then, for $\fm$-almost every $x \in X$,
we deduce from Theorem~\ref{th:RLF}(iii) that
\[ \langle \nabla(g \circ \Phi),\nabla u \rangle_X \big( {\tilde{F}}_t(x) \big)
 =-\frac{d}{dt} \big[ (g \circ \Phi) \big( {\tilde{F}}_t(x) \big) \big]
 =-\frac{d}{dt} \big[ \tilde{g}(x) \big] =0 \]
in the distributional sense in $t \in \R$,
where $\tilde{g} \in W^{1,2} \cap L^{\infty}(Y)$ satisfies $g(y,t)=\tilde{g}(y)$.
(To be precise, we cut-off $g \circ \Phi$ to be in $W^{1,2}(X)$
when we apply Theorem~\ref{th:RLF}(iii).)
This completes the proof of \eqref{eq:E=E2} and then the claim.
$\qedd$
\end{proof}

\begin{theorem}[Isometric splitting]\label{th:isom}
The maps $\Phi$ and $\Psi$ are isometric.
\end{theorem}

\begin{proof}
This is a consequence of the energy identity in Proposition~\ref{pr:E=E}
and \cite[Proposition~4.20]{Gi-split}.
Recall that the Sobolev-to-Lipschitz property, 
which is required in the cited proposition, 
holds on $\RCD (K,\infty)$-spaces 
as we mentioned in \S \ref{ssc:RCD}.
$\qedd$
\end{proof}

\begin{remark}\label{rm:affine}
The discussions in Sections~\ref{sc:affine}--\ref{sc:isom}
could be compared with the study of spaces admitting nonconstant affine functions.
The existence of a nonconstant affine function is a strong constraint
and forces the space to possess some splitting phenomenon.
See \cite{In,Ma,AB,HL} for related results concerning affine functions
on Riemannian manifolds or metric spaces,
and \cite{Oh-tg,Ly,BMS} for further studies on affine maps between (or into) metric spaces.
\end{remark}

\section{Final step and some remarks}\label{sc:final}%%%%%%%%%%
%%%%%%%%%%%%%%%%%%%%%

We finish the proof of Theorem~\ref{th:main} by iteration.
The case of $k=1$ was shown by the previous step.
If $k \ge 2$, then the space $(Y,d_Y,\fm_Y)$ has $\lambda_1=1$
and splits off the $1$-dimensional Gaussian space.
We iterate this procedure and complete the proof.
$\qedd$
\smallskip

We close the article with several remarks.

\begin{remark}\label{rm:final}
(a)
It is somewhat implicit in our discussion that the sharp spectral gap prevents
spaces ``with boundary'' such as $Y \times [0,\infty)$ showing up
(while $Y$ of $Y \times \R$ can have a boundary).
Indeed, on $Y \times [0,\infty)$, the function $u(y,t)=t$ is not an eigenfunction
since its measure-valued Laplacian has singularity on $Y \times \{0\}$.

(b)
It is well-known that a rigidity result for a compact family of spaces (in a certain topology)
can be used to show the corresponding \emph{almost rigidity}.
See \cite{Gi-split} for the case of almost splitting theorem.
The compactness, however, fails for the class of $\RCD(K,\infty)$-spaces even when $K>0$.
This is another difficulty due to the lack of the doubling condition.
We know (at least) two kinds of examples of sequences of $\RCD(1,\infty)$-spaces
having no convergent subsequence.
Firstly, the sequence of Gaussian spaces
\[ (X_n,d_n,\fm_n) := (\R^n,|\cdot|,\e^{-|x|^2/2} dx^1 dx^2 \cdots dx^n),
 \quad n \in \N, \]
consists of $\RCD(1,\infty)$-spaces and has no convergent subsequence
in the sense of the \emph{measured Gromov--Hausdorff convergence}
nor of the \emph{measured Gromov convergence}
(see \cite[Corollary~7.42]{Shi} and \cite{GMS} for details).
Secondly, the sequence
\[ (Z_k,d_k,\fm_k) := (\R^2,|\cdot|,\e^{-(kx^2+y^2)/2} dxdy),
 \quad k \in \N, \]
also consists of $\RCD(1,\infty)$-spaces and has no convergent subsequence
in the measured Gromov--Hausdorff topology.
This sequence, however, converges to $(\R,|\cdot|,\e^{-y^2/2}dy)$
in the weaker notion of the measured Gromov topology.
We remark that, in either case, the sharp spectral gap is attained
($\lambda_1(X_n)=\lambda_1(Z_k)=1$ for all $n,k$).

(c)
The Lichnerowicz inequality $\lambda_1 \ge KN/(N-1)$ under the bound
$\Ric_N \ge K>0$ holds true also for the ``negative effective dimension'' $N<0$,
see \cite{KM,Oh-neg}.
It would be worthwhile to consider the rigidity problem on this widely open situation.

(d)
Another possible generalization is the case of Finsler manifolds
(or more generally $\CD(K,\infty)$-spaces),
where the spectral gap and a Cheeger--Gromoll type splitting theorem are known
(\cite{Oh-int,Oh-split}).
We refer to \cite[Theorem~8.1]{Ke-obata} for the case of the Lichnerowicz inequality ($N>1$).

(e) 
In \cite{AM:gi} the authors prove a sharp Gaussian isoperimetric inequality for $\RCD(K,\infty)$-spaces with $K>0$, which generalizes
the L\'evy isoperimetric inequality. We expect that equality in this result should yield 
the same rigidity statement as in this paper, similar to the finite dimensional situation of the L\'evy-Gromov isoperimetric inequality \cite{CM}.
\end{remark}

{\small%%%

}


\begin{thebibliography}{AGMR}%%%%%%%%%%%%%%%%%%%%

\bibitem[AB]{AB}
S.~B.~Alexander and R.~L.~Bishop,
A cone splitting theorem for Alexandrov spaces.
Pacific J.\ Math.\ {\bf 218} (2005), 1--15.

\bibitem[AG]{AG}
L.~Ambrosio, N.~Gigli
A user's guide to optimal transport.
Modelling and Optimisation of Flows on Networks (2013).

\bibitem[AGMR]{AGMR}
L.~Ambrosio, N.~Gigli, A.~Mondino and T.~Rajala,
Riemannian Ricci curvature lower bounds in metric measure spaces with $\sigma$-finite measure.
Trans.\ Amer.\ Math.\ Soc.\ {\bf 367} (2015), 4661--4701.

\bibitem[AGS1]{AGS-book}
L.~Ambrosio, N.~Gigli and G.~Savar\'e,
Gradient flows in metric spaces and in the space of probability measures.
Birkh\"auser Verlag, Basel, 2005.

\bibitem[AGS2]{AGS-hf}
L.~Ambrosio, N.~Gigli and G.~Savar\'e,
Calculus and heat flow in metric measure spaces and applications to spaces with Ricci bounds from below.
Invent.\ Math.\ {\bf 195} (2014), 289--391.

\bibitem[AGS3]{AGS-rcd}
L.~Ambrosio, N.~Gigli and G.~Savar\'e,
Metric measure spaces with Riemannian Ricci curvature bounded from below.
Duke Math.\ J.\ \textbf{163} (2014), 1405--1490.

\bibitem[AGS4]{AGS-be}
L.~Ambrosio, N.~Gigli and G.~Savar\'e,
Bakry--\'Emery curvature-dimension condition and Riemannian Ricci curvature bounds.
Ann.\ Probab.\ {\bf 43} (2015), 339--404.

\bibitem[AM]{AM:gi}
L.~Ambrosio and A.~Mondino,
Gaussian-type isoperimetric inequalities in $\RCD(K,\infty)$ probability spaces for positive $K$. 
Atti Accad. Naz. Lincei Rend. Lincei Mat. Appl. 27 (2016), no. 4, 497--514.

\bibitem[AMS1]{ams}
L.~Ambrosio, A.~Mondino and G.~Savar\'e,
On the Bakry--\'Emery condition, the gradient estimates and the local-to-global property of
$\RCD^*(K,N)$ metric measure spaces.
J.\ Geom.\ Anal.\ {\bf 26} (2016), 24--56.

\bibitem[AMS2]{AMS2}
L.~Ambrosio, A.~Mondino and G.~Savar\'e,
Nonlinear diffusion equations and curvature conditions in metric measure spaces.
Preprint (2016). Available at {\sf arXiv:1509.07273}

\bibitem[AT1]{AT}
L.~Ambrosio and D.~Trevisan,
Well-posedness of Lagrangian flows and continuity equations in metric measure spaces.
Anal.\ PDE {\bf 7} (2014), 1179--1234.

\bibitem[AT2]{AT2}
L.~Ambrosio and D.~Trevisan,
Lecture notes on the DiPerna--Lions theory in abstract measure spaces.
Preprint (2015). Available at {\sf arXiv:1505.05292}

\bibitem[Ba]{Ba1}
D.~Bakry,
Transformations de Riesz pour les semi-groupes sym\'etriques.~II.
\'Etude sous la condition $\Gamma_2 \ge 0$. (French)
S\'eminaire de probabilit\'es, {\bf XIX}, 1983/84, 145--174, 
Lecture Notes in Math., {\bf 1123}, Springer, Berlin, 1985.

%\bibitem[Ba2]{Ba2}
%D.~Bakry,
%L'hypercontractivit\'e et son utilisation en th\'eorie des semigroupes. (French)
%Lectures on probability theory (Saint-Flour, 1992), 1--114,
%Lecture Notes in Math., {\bf 1581}, Springer, Berlin, 1994.

\bibitem[BE]{BE}
D.~Bakry and M.~\'Emery, Diffusions hypercontractives. (French)
S\'eminaire de probabilit\'es, XIX, 1983/84, 177--206,
Lecture Notes in Math., {\bf 1123}, Springer, Berlin, 1985.

%\bibitem[BGL]{BGL}
%D.~Bakry, I.~Gentil and M.~Ledoux,
%Analysis and geometry of Markov diffusion operators.
%Springer, Cham, 2014.

\bibitem[BMS]{BMS}
H.~Bennett, C.~Mooney and R.~Spatzier,
Affine maps between CAT$(0)$ spaces.
Geom.\ Dedicata {\bf 180} (2016), 1--16.

\bibitem[BH]{BH}
N.~Bouleau and F.~Hirsch, Dirichlet forms and analysis on Wiener space.
%De Gruyter Studies in Mathematics, {\bf 14}.
Walter de Gruyter \& Co., Berlin, 1991.

\bibitem[CMi]{CMi}
F.~Cavalletti and E.~Milman,
The globalization theorem for the curvature dimension condition.
Preprint (2016). Available at {\sf arXiv:1612.07623}

\bibitem[CMo]{CM}
F.~Cavalletti and A.~Mondino,
Sharp and rigid isoperimetric inequalities in metric-measure spaces with lower Ricci curvature bounds.
Invent.\ Math.\ (to appear). Available at {\sf arXiv:1502.06465}

\bibitem[Ch]{Ch}
J.~Cheeger,
Differentiability of Lipschitz functions on metric measure spaces.
Geom.\ Funct.\ Anal.\ {\bf 9} (1999), 428--517.

\bibitem[CG]{CG}
J.~Cheeger and D.~Gromoll,
The splitting theorem for manifolds of nonnegative Ricci curvature.
J.\ Differential Geometry {\bf 6} (1971/72), 119--128.

\bibitem[CZ]{CZ}
X.~Cheng and D.~Zhou,
Eigenvalues of the drifted Laplacian on complete metric measure spaces.
Commun.\ Contemp.\ Math.\ {\bf 19} (2017), 1650001, 17 pp.

\bibitem[DL]{DL}
R.~J.~DiPerna and P.-L.~Lions,
Ordinary differential equations, transport theory and Sobolev spaces.
Invent.\ Math.\ {\bf 98} (1989), 511--547.

\bibitem[EKS]{EKS}
M~Erbar, K.~Kuwada and K.-T.~Sturm,
On the equivalence of the entropic curvature-dimension condition
and Bochner's inequality on metric measure spaces.
Invent.\ Math.\ {\bf 201} (2015), 993--1071.

\bibitem[FOT]{FOT}
M.~Fukushima, Y.~Oshima and M.~Takeda,
Dirichlet forms and symmetric Markov processes.
Second revised and extended edition.
de Gruyter Studies in Mathematics, {\bf 19}.
Walter de Gruyter \& Co., Berlin, 2011.

\bibitem[Gi1]{Gi-Ondiff}
N.~Gigli, On the differential structure of metric measure spaces and applications. 
Mem.\ Amer.\ Math.\ Soc.\ {\bf 236} (2015).

\bibitem[Gi2]{Gi-split}
N.~Gigli, The splitting theorem in non-smooth context.
Preprint (2013). Available at {\sf arXiv:1302.5555}

\bibitem[Gi3]{Gi-ov}
N.~Gigli,
An overview of the proof of the splitting theorem in spaces with non-negative Ricci curvature.
Anal.\ Geom.\ Metr.\ Spaces {\bf 2} (2014), 169--213.

\bibitem[Gi4]{Gi-dg}
N.~Gigli,
Nonsmooth differential geometry -- An approach tailored for spaces with Ricci curvature bounded from below.
Mem.\ Amer.\ Math.\ Soc.\ (to appear). Available at {\sf arXiv:1407.0809}

\bibitem[GKO]{GKO}
N.~Gigli, K.~Kuwada and S.~Ohta, Heat flow on Alexandrov spaces.
Comm.\ Pure Appl.\ Math.\ {\bf 66} (2013), 307--331.

\bibitem[GMS]{GMS}
N.~Gigli, A.~Mondino and G.~Savar\'e,
Convergence of pointed non-compact metric measure spaces and stability of Ricci curvature bounds and heat flows.
Proc.\ Lond.\ Math.\ Soc.\ (3) {\bf 111} (2015), 1071--1129.

\bibitem[GT]{GT}
N.~Gigli and L.~Tamanini,
Second order differentiation formula on compact $\RCD^*(K,N)$ spaces 
Preprint (2017), Available at {\sf arXiv:1701.03932}.

\bibitem[H]{H}
BX.~Han,
Ricci Tensor on $\RCD^*(K,N)$ Spaces
J.\ Geom.\ Anal. (2017). doi:10.1007/s12220-017-9863-7

\bibitem[GH]{GH}
N.~Gigli and BX.~Han
Sobolev Spaces on Warped Products
Preprint (2015), Available at {\sf arXiv:1512.03177}.

\bibitem[HN]{HN}
H.-J.~Hein and A.~Naber,
New logarithmic Sobolev inequalities and an $\epsilon$-regularity theorem for the Ricci flow.
Comm.\ Pure Appl.\ Math.\ {\bf 67} (2014), 1543--1561.

\bibitem[HL]{HL}
P.~Hitzelberger and A.~Lytchak,
Spaces with many affine functions.
Proc.\ Amer.\ Math.\ Soc.\ {\bf 135} (2007), 2263--2271.

\bibitem[In]{In}
N.~Innami, Splitting theorems of Riemannian manifolds.
Compositio Math.\ {\bf 47} (1982), 237--247.

%\bibitem[Ke1]{Ke-diam}
%C.~Ketterer, Cones over metric measure spaces and the maximal diameter theorem.
%J.\ Math.\ Pures Appl.\ (9) {\bf 103} (2015), 1228--1275.

\bibitem[Ke]{Ke-obata}
C.~Ketterer, Obata's rigidity theorem for metric measure spaces.
Anal.\ Geom.\ Metr.\ Spaces {\bf 3} (2015), 278--295.

\bibitem[KM]{KM}
A.~V.~Kolesnikov and E.~Milman,
Poincar\'e and Brunn--Minkowski inequalities on weighted Riemannian manifolds with boundary.
Preprint (2013). Available at {\sf arXiv:1310.2526}

% \bibitem[KZ]{KZ}
% {P.~Koskela and Y.~Zhou}, 
% Geometry and analysis of Dirichlet forms. 
% Adv.\ Math.\ {\bf 231} (2012), 2755--2801.

\bibitem[Ku]{Ku}
K.~Kuwada, A probabilistic approach to the maximal diameter theorem.
Math.\ Nachr.\ {\bf 286} (2013), 374--378.

%\bibitem[Li]{Li}
%S.~Lisini,
%Characterization of absolutely continuous curves in Wasserstein spaces.
%Calc.\ Var.\ Partial Differential Equations {\bf 28} (2007), 85--120.

\bibitem[LV]{LV}
J.~Lott and C.~Villani,
Ricci curvature for metric-measure spaces via optimal transport.
Ann.\ of Math.\ {\bf 169} (2009), 903--991.

\bibitem[LV]{LV-de}
J.~Lott and C.~Villani,
Weak curvature conditions and functional inequalities. 
J. Funct. Anal. 245 (2007), no. 1, 311--333.

\bibitem[Ly]{Ly}
A.~Lytchak, Affine images of Riemannian manifolds.
Math.\ Z.\ {\bf 270} (2012), 809--817.

\bibitem[Mai]{Mai}
C.~H.~Mai,
On Riemannian manifolds with positive weighted Ricci curvature of negative effective dimension.
Preprint (2017). Available at {\sf arXiv:1704.06091}

\bibitem[Ma]{Ma}
Y.~Mashiko, A splitting theorem for Alexandrov spaces.
Pacific J.\ Math.\ {\bf 204} (2002), 445--458.

\bibitem[Mc]{Mc}
R.~J.~McCann,
Polar factorization of maps on Riemannian manifolds.
Geom.\ Funct.\ Anal.\ {\bf 11} (2001), 589--608.

\bibitem[Ob]{Ob}
M.~Obata,
Certain conditions for a Riemannian manifold to be isometric with a sphere.
J.\ Math.\ Soc.\ Japan {\bf 14} (1962), 333--340.

\bibitem[Oh1]{Oh-tg}
S.~Ohta, Totally geodesic maps into metric spaces.
Math.\ Z.\ {\bf 244} (2003), 47--65.

\bibitem[Oh2]{Oh-int}
S.~Ohta, Finsler interpolation inequalities.
Calc.\ Var.\ Partial Differential Equations {\bf 36} (2009), 211--249.

\bibitem[Oh3]{Oh-split}
S.~Ohta,
Splitting theorems for Finsler manifolds of nonnegative Ricci curvature.
J.\ Reine Angew.\ Math.\ {\bf 700} (2015), 155--174.

\bibitem[Oh4]{Oh-neg}
S.~Ohta,
$(K,N)$-convexity and the curvature-dimension condition for negative $N$.
J.\ Geom.\ Anal.\ {\bf 26} (2016), 2067--2096.

\bibitem[Oh5]{Oh-needle}
S.~Ohta, Needle decompositions and isoperimetric inequalities in Finsler geometry.
J.\ Math.\ Soc.\ Japan (to appear). Available at {\sf arXiv:1506.05876}

\bibitem[Oh6]{Oh-nlga}
S.~Ohta, Nonlinear geometric analysis on Finsler manifolds.
Eur.\ J.\ Math.\ (to appear). Available at {\sf arXiv:1704.01257}

\bibitem[OS]{OS-nc}
S.~Ohta and K.-T.~Sturm, Non-contraction of heat flow on Minkowski spaces.
Arch.\ Ration.\ Mech.\ Anal.\ {\bf 204} (2012), 917--944.

\bibitem[Ra]{Ra}
T. Rajala, Interpolated measures with bounded density in metric spaces
satisfying the curvature-dimension conditions of Sturm.
J.\ Funct.\ Anal.\ {\bf 263} (2012), 896--924.

\bibitem[vRS]{vRS}
M.-K.~von Renesse and K.-T.~Sturm,
Transport inequalities, gradient estimates, entropy and Ricci curvature.
Comm.\ Pure Appl.\ Math.\ {\bf 58} (2005), 923--940.

\bibitem[Sa]{Sa}
G.~Savar\'e, Self-improvement of the Bakry--\'Emery condition
and Wasserstein contraction of the heat flow in $\RCD(K,\infty)$ metric measure spaces.
Discrete Contin.\ Dyn.\ Syst.\ {\bf 34} (2014), 1641--1661.

\bibitem[Sha]{Sha}
N.~Shanmugalingam,
Newtonian spaces: an extension of {S}obolev spaces to metric measure spaces
Rev. Mat. Iberoamericana, \ {\bf 16} (2000), 243--279.

\bibitem[Sh]{Shi}
T.~Shioya,
Metric measure geometry. 
Gromov's theory of convergence and concentration of metrics and measures.
EMS Publishing House, Z\"urich, 2016.


\bibitem[St1]{StI}
K.-T.~Sturm, On the geometry of metric measure spaces.~I.
Acta Math.\ {\bf 196} (2006), 65--131.

\bibitem[St2]{StII}
K.-T.~Sturm, On the geometry of metric measure spaces.~II.
Acta Math.\ {\bf 196} (2006), 133--177.

\bibitem[St3]{St-gf}
K.-T.~Sturm,
Gradient flows for semiconvex functions on metric measure spaces -- existence, uniqueness and Lipschitz continuity.
Preprint (2014). Available at {\sf arXiv:1410.3966}

\bibitem[Vi]{Vi}
C.~Villani, Optimal transport, old and new.
Springer-Verlag, Berlin, 2009.

\end{thebibliography}
\end{document}